\documentclass[a4paper,12pt]{article}
\title{Motivic zeta function of the Hilbert schemes of points on a surface}
\author{Luigi Pagano}

\usepackage{CommandsAndStyle}
\usepackage{titlesec}

\titleformat{\subsubsection}[runin]{\bfseries}{\thesubsubsection}{1em}{}[]

\begin{document}

\maketitle

\abstract{Let $K$ be a discretely-valued field. Let $X\rightarrow \Spec K$ be a surface with trivial canonical bundle. In this paper we construct a weak Néron model of the schemes $\Hilb^n(X)$ over the ring of integers $R\subseteq K$. We exploit this construction in order to compute the Motivic Zeta Function of $\Hilb^n(X)$ in terms of $Z_X$. We determine the poles of $Z_{\Hilb^n(X)}$ and study its monodromy property, showing that if the monodromy conjecture holds for $X$ then it holds for $\Hilb^n(X)$ too. 

\begin{center}
{\bfseries Excerptum}
\end{center}

Sit $K$ corpus cum absoluto ualore discreto. Sit $ X\rightarrow \Spec K$ leuigata superficies cum canonico fasce congruenti $\pazocal{O}_X$. In hoc scripto defecta Neroniensia paradigmata $\Hilb^n(X)$ schematum super annulo integrorum in $K$ corpo, $R \subset K$, constituimus. Ex hoc, Functionem Zetam Motiuicam $Z_{\Hilb^n(X)}$, dato $Z_X$, computamus. Suos polos statuimus et suam monodromicam proprietatem studemus, coniectura monodromica, quae super $X$ ualet, ualere super $\Hilb^n(X)$ quoque demostrando.}

\section*{Introduction}

The Motivic Zeta Function, introduced by Denef and Loeser, is a powerful invariant attached to hypersurface singularities. This invariant is a formal series with coefficients in a Grothendieck ring of varieties which has been proved to be rational in \cite{NS07}, in a sense that will become clear in \S \ref{sect:intrmonconj}. The most discussed open question concerning the motivic zeta function is the \emph{Monodromy Conjecture} which claims the existence of a relationship between the poles of the function and the eigenvalues of the local monodromy operator associated to the hypersurface singularity. In this paper we focus on a natural counterpart, introduced in \cite{LS}, of Denef's and Loeser's zeta function, whose argument consists of a punctured neighbourhood of the hypersurface singularity, together with a volume form on such neighbourhood, rather than involving the hypersurface itself. The two invariants are closely related and the questions about the Monodromy Conjecture make sense in both cases. 

Let $K$ be a complete discretely valued field. Let $R$ be its valuation ring and let $k$ be its residue field, which we assume to be algebraically closed.  Let $X\rightarrow \Spec K$ be a smooth algebraic variety. This datum is the arithmetic analog of the datum of a family of complex varieties over a small punctured disk, where one usually uses the fundamental group of the punctured disk to give a monodromy action on the cohomology. In our case the absolute Galois group of $K$, being canonically isomorphic to the étale fundamental group of $\Spec K$, replaces the usual fundamental group and acts on $X$, inducing an action on its cohomology: $H^\ast(X_{\overline{K}}, \Q_l)$, where $l$ is any prime number coprime with $\chara k$. We will assume that the wild inertia subgroup acts trivially on $X$, so that the action of $\Gal(\overline{K}|K)$ can be identified with the action of its tame quotient, which admits a single topological generator $\sigma$, hence we can study the monodromy action by focusing on the action of the linear operator $\sigma^\ast\colon H^\ast (X_{\overline{K}}, \Q_l)\rightarrow H^\ast(X_{\overline{K}}, \Q_l)$. This operator is quasi-unipotent, meaning that $\exists a,b\in \N$ such that ${\left( (\sigma^{\ast})^a-\id\right)}^b=0$, hence its eigenvalues, often called \emph{monodromy eigenvalues}, are roots of the unity. 

Now assume $X$ is a Calabi-Yau variety, \emph{i.e.} a smooth algebraic variety whose canonical sheaf $\omega_X\coloneqq \Omega_{X/K}^{\dim_X}$ is a trivial line bundle and let $\omega$ be a volume form on $X$. Using this datum, we denote by $Z_{X,\omega}(T)\in \mathcal{M}_k[[T]]$, the Motivic Zeta Function of Definition \ref{def:motzeta}. This invariant is supposed to be closely related to the monodromy operator described above in a sense explained in the following:
\begin{conjecture*}[Monodromy conjecture]
Let $q\in\Q$ be a pole of $Z_{X,\omega}(T)$, then $e^{2\pi i q}$ is a monodromy eigenvalue.
\end{conjecture*} 
The monodromy conjecture has been proven in several classes of varieties: Halle and Nicaise proved it for Abelian varieties, \cite{HN_Ab}, while Jaspers proved it in \cite{Jas} when $X$ is a K3 surfaces admitting a Crauder-Morrison model and Overkamp proved it for Kummer K3 surfaces in \cite{Ove}. Yet we do not know whether all the K3 surfaces satisfy the Monodromy conjecture. 

In this paper we study this conjecture for Hilbert schemes of points on surfaces; we report here a simplified version of our main result, Theorem \ref{thm:monconjHn}; a more precise statement is provided in \S \ref{sect:proofMC}, after having introduced the notions involved:
\begin{theorem*}
Let $X$ be a surface with trivial canonical bundle satisfying the monodromy conjecture. 

Then the conjecture holds also for $\Hilb^{n}(X)$, $\forall n\in \N$.
\end{theorem*} 
\begin{proof}[Sketch of the proof] We will outline here the main ideas involved in the proof.

The computation of the MZF relies on the tools of motivic integration, thus on the construction of a weak Néron model of $\Hilb^n(X)$. After fixing an \emph{sncd} model of $X$ over the ring $R$ (satisfying a technical assumption introduced in \S \ref{subsect:constwnm}), call $\mathfrak{X}$ such a model, then for all integers $m$ divisible enough and coprime with $\chara K$, it is possible to construct semistable models $\mathfrak{X}(m)\rightarrow \Delta(m) =\Spec R(m)$ of $X\times_K K(m)$, where $K(m)/K$ is the unique extension of degree $m$ such that the integral closure $R(m)= \overline{R} \subseteq K(m)$ is totally ramified over $R$. We give an explicit construction of such models, see \S \ref{sect:equivsemsred}, satisfying the following additonal properties:
\begin{itemize}
\item The natural action of $\mu_m=\Gal(K(m)|K)$ on $X(m)$ extends to $\mathfrak{X}(m)$;
\item For all points $x\in \mathfrak{X}(m)_{k,\sm}$, the stabilizer $\Stab_x\subseteq \mu_m$ acts trivially on the whole irreducible component $E\subseteq \mathfrak{X}(m)_{k}$ containing $x$.
\end{itemize}
The relative Hilbert functor provides a proper scheme $\Hilb^n(\mathfrak{X}(m)/ \Delta(m))\rightarrow \Delta(m)$ whose generic fibre is $\Hilb^n(X(m))$. This model inherits from $\mathfrak{X}(m)$ an action of $\mu_m$, in particular it is a proper, equivariant model of $X(m)$ over $\Delta(m)$. If we remove the critical points of the map $\mathfrak{X}(m)\rightarrow \Delta(m)$, then we repeat the construction obtaining a smooth $\mu_m-$equivariant subscheme $\Hilb^n(\mathfrak{X}(m)_{\sm}/ \Delta(m))\subseteq \Hilb^n(\mathfrak{X}(m)/ \Delta(m))$. By the means of Weil restriction of scalars (see \S \ref{sect:wnm}), we construct a smooth model of $\Hilb^n(X)$ over $R$, namely $ {\left( \Res_{\Delta(\tilde{n})/\Delta} \Hilb^n(\mathfrak{X}(m)_{\sm} / \Delta(\tilde{n}))\right)}^{\mu_{\tilde{n}}}$. 

Under a suitable choice of $m$ ($m=\tilde{n}$ under the notation of \S \ref{sect:motintHilb}), the model above satisfies the weak extension property with respect to $\Spec K\subseteq \Spec R$, hence it is a weak Néron model of $\Hilb^n(X)$. This model provides a direct way to compute the motivic integral $\int_{\Hilb^n(X)}\omega^{[n]}$; choosing $m=m_1\tilde{n}$ allows us to compute the motivic integral of the Hilbert schemes \\ $\Hilb^n(X(m_1)/K(m_1))$, for $m_1\in \N_+$, therefore we get the formal series defining $Z_{\Hilb^n(X), \omega^{[n]}}$, since $\Hilb^n(X)\times_K K(m)= \Hilb^n(X(m))$. The value of the Zeta function is, indeed, written implicitly in Theorem \ref{thm:implicitZ} and explicitly in Equation \eqref{eqn:explZ}, i.e. 
$$
Z_{\Hilb^n(X), \omega^{[n]}} = \sum_{\alpha \dashv n} \prod_{j=1}^\infty  \left(  \L^{(j-1)\alpha_j} \Sym^{\alpha_j}\left( Z_{X(j), \omega(j)} \right) \right)  \punto
$$
Using this explicit formula and the results developed in \S \ref{sect:studysym}, we find out that every pole of $Z_{\Hilb^n(X), \omega^{[n]}}$ is the sum of $n$ poles of $Z_{X, \omega}$, possibly repeated, we do not know yet if the converse is true, i.e. whether every possible sum of $n$ poles of $Z_{X, \omega}$ is a pole of $Z_{\Hilb^n(X), \omega^{[n]}}$. Putting the description of the poles of the Zeta Function together with the description of the cohomology of $\Hilb^n(X)$ in \cite{GS}, one finds out the claimed result about the monodromy conjecture.
\end{proof} 
A further byproduct of the computations we carried out in \S \ref{sect:forser}, we also achieved Proposition \ref{prop:monconjprod}:
\begin{proposition*}
Let $Y, Z$ be two Calabi-Yau varieties endowed with volume forms $\omega_1, \omega_2$ satisfying the monodromy conjecture. Let $\omega$ be volume forms on $Y \times Z$ the volume form $\omega\coloneqq \pr_Y^\ast \omega_1 \wedge \pr_Z^\ast \omega_2$. Then also $Y\times Z$, endowed with the volume form $\omega$, satisfies the monodromy conjecture.
\end{proposition*}
We conclude the paper by discussing the Monodromy conjecture in an explicit example, i.e. comparing the poles of the zeta functions of a quartic surface embedded in $\P^3_K$ and of its Hilbert schemes of degree 2.

\subsection*{Organization of the paper}

In \S \ref{sect:GRV} we introduce the Grothendieck rings of varieties and their equivariant versions, using them to define the main actors of the manuscript. 

In \S \ref{sect:wnm} we recall the notion of a weak Néron model and show some techniques that we will use in order to construct them. 

In \S \ref{sect:equivsemsred} we adapt the theory of potential semistable reduction of families of surfaces in our case. We construct a specific equivariant semistable model satisfying a good property with respect to the Galois action of the extension. 

In \S \ref{sect:intrmonconj} we recall the notion of motivic integration and the definition of the Motivic Zeta Function. Then we explain the notion of rationality for power series with coefficients in $\mathcal{M}_k$, in $\mathcal{M}_k\left[ (\L^r-1)^{-1} \colon 0<r\in \N \right]$ and in $\widehat{\mathcal{M}_k}$, giving also a definition of a pole of such functions. We conclude by stating the Monodromy Conjecture in two forms. 

In \S \ref{sect:forser} we discuss some facts concerning the poles of rational functions with coefficients in our motivic ring. These properties will be useful when applied to the formula that we will produce for the motivic zeta function of Hilbert schemes. 

In \S \ref{sect:motintHilb} we give a construction of the weak Néron models of Hilbert schemes of points on surfaces and use those models to compute their motivic zeta function.

We finally discuss the monodromy property in \S \ref{sect:proofMC}.



\subsection*{Acknowledgements}

The author is very grateful to his PhD advisor, Prof. Lars Halle, for introducing him to the problem and for many interesting ideas and conversations. The author is also grateful to Prof. Johannes Nicaise for many helpful conversations.

\vspace{\baselineskip}
\noindent
\framebox[\textwidth]{
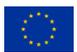
\begin{tabular*}{0.96\textwidth}{@{\extracolsep{\fill} }cp{0.84\textwidth}}
\raisebox{-0.7\height}{%
    \begin{tikzpicture}[y=0.80pt, x=0.8pt, yscale=-1, inner sep=0pt, outer sep=0pt, 
    scale=0.12]
    \definecolor{c003399}{RGB}{0,51,153}
    \definecolor{cffcc00}{RGB}{255,204,0}
    \begin{scope}[shift={(0,-872.36218)}]
      \path[shift={(0,872.36218)},fill=c003399,nonzero rule] (0.0000,0.0000) rectangle (270.0000,180.0000);
      \foreach \myshift in 
           {(0,812.36218), (0,932.36218), 
    (60.0,872.36218), (-60.0,872.36218), 
    (30.0,820.36218), (-30.0,820.36218),
    (30.0,924.36218), (-30.0,924.36218),
    (-52.0,842.36218), (52.0,842.36218), 
    (52.0,902.36218), (-52.0,902.36218)}
        \path[shift=\myshift,fill=cffcc00,nonzero rule] (135.0000,80.0000) -- (137.2453,86.9096) -- (144.5106,86.9098) -- (138.6330,91.1804) -- (140.8778,98.0902) -- (135.0000,93.8200) -- (129.1222,98.0902) -- (131.3670,91.1804) -- (125.4894,86.9098) -- (132.7547,86.9096) -- cycle;
    \end{scope}
    \end{tikzpicture}%
}
&
This project has received funding from the European Union Horizon 2020 research and
innovation programme under the Marie Sk{\l}odowska-Curie grant agreement No 801199.
\end{tabular*}
}

\section{Grothendick ring of varieties}
\label{sect:GRV}

In this section we introduce the rings containing the coefficients of the formal series we shall study later on. 

\subsection{A motivic ring}

\subsubsection{} Fix a field $k$ and consider the category of algebraic varieties $\Var_k$. Let $K_0(\Var_k)$ be the group whose generators are isomorphism classes in $\Var_k$ and whose relations, called \emph{scissor relations}, are generated by elements in the form 
$$ X - Y - (X\backslash Y) \virgola $$
whenever $X$ is an algebraic variety and $Y\subseteq X$ is a closed subscheme. We denote by $[X]$ the class of $X\in \Var_k$ in $K_0(\Var_k)$.

\subsubsection{} There is a unique ring structure on $K_0(\Var_k)$ such that for all $X,Y\in \Var_k$ one has $[X]\cdot [Y] = [X\times_k Y]$. With this ring structure, $K_0(\Var_k)$ is called \emph{the Grothendieck ring of varieties}. It is also characterized by the following universal property:

\paragraph{Universal property of $K_0(\Var_k)$.} Let $R$ be a ring and let $\Psi\colon \Var_k\rightarrow R$ a multiplicative and additive invariant, i.e. a function which associate to a variety $X$ an element $\Psi (X)\in R$ such that $\Psi(X\times_k Y)=\Psi(X)\Psi(Y)$ and if $X= Y\cup Z$, then $\Psi(X)+\Psi(Y\cap Z)=\Psi(Y)+\Psi(Z)$.
Then there is a unique ring homomorphism $\varphi\colon K_0(\Var_k)\rightarrow R$ such that $$\forall X\in \Var_k, \mbox{ one has that }  \; \varphi([X])=\Psi(X) \punto $$
Due to this universal property, $K_0(\Var_k)$ naturally embeds in the ring of Chow motives.


\subsection{Localised Grothendieck ring}

\subsubsection{} A ring that is worth some consideration is obtained as a localization of $K_0(\Var_k)$.

\begin{definition}[Localised Grothendieck ring of varieties]
Let us denote by $\L$ the element $[\A^1_k]\in K_0(\Var_k)$. 

The localisation of $K_0(\Var_k)$ with respect to $\L$,
$$ \mathcal{M}_k \coloneqq K_0(\Var_k)[\L^{-1}]  \virgola $$
is called \emph{the localised Grothendieck ring of varieties}.
\end{definition}

\subsubsection{} By combining the universal property of localization and of $K_0(\Var_k)$, one can define $\mathcal{M}_k$ as a universal ring for all the invariants $\Var_k\rightarrow R$ which send $\A_k^1$ in $R^\ast$.



\subsection{Completed Grothendieck ring}

\subsubsection{} Consider the filtration $\pazocal{F}^n\mathcal{M}_k\coloneqq \left\langle \L^r [X] | r\in \Z\virgola \dim[X]+r\leq - n \right\rangle_\Z$. 

\begin{definition}[Completed Grothendieck ring of varieties]
The completed Grothendieck ring of varieties $\widehat{\mathcal{M}_k}$ is the completion of $\mathcal{M}_k$ with respect to the filtration $\pazocal{F}^\bullet$.
\end{definition}


\subsection{Equivariant setting}
\label{subsect:equivring}

\subsubsection{} All the three rings above have an equivariant version, i.e. can be constructed in the category $\Var_k^G$ of algebraic varieties endowed with the action of an algebraic group $G$.

\subsubsection{} As before, the \emph{equivariant Grothendieck group of varieties} $K_0(\Var_k^G)$ is the group generated on the isomorphism classes of $G-$varieties with relation of two kinds:
\begin{description}
\item[Scissor relations] Let $X$ be a $G-$variety and $Y$ a $G-$invariant closed subscheme, then
$$ X - Y - (X\backslash Y) \virgola $$
is $0$ in $K_0(\Var^G_k)$.
\item[Trivializing relations]
Let $S\in \Var_k^G$ and let $V\rightarrow S$ be a $G-$equivariant affine bundle of rank $d$. Then 
$$ V- (S\times \A_k^d) \in \Var_k^G $$ 
is set to $0$, where $G$ acts trivially on the second factor of $S\times \A_k^d$.
\end{description}
There is a unique ring structure on $K_0(\Var_k^G)$ such that for every two $X, Y\in \Var_k^G$, we have $[X]\cdot [Y]\coloneqq [X\times_k Y]$, where the action of $G$ on $X\times_k Y$ is the diagonal action.

\subsubsection{} We use again the symbol $\L\coloneqq [\A_k^1]$, where the group $G$ acts trivially on the affine line; thus the trivializing relations tell us nothing more than:
$$ [V]=\L^d[S]  \virgola $$
whenever $V\rightarrow S$ is an equivariant affine bundle of rank $d$.

\subsubsection{} Similarly to what we did in the previous sections, we define the localisation \\ $\mathcal{M}_k^G\coloneqq K_0(\Var_k^G)[\L^{-1}]$ and its completion $\widehat{\mathcal{M}_k^G}$ with respect to the filtration $$\pazocal{F}^n\mathcal{M}_k\coloneqq \left\langle \L^r [X] | r\in \Z\virgola \dim[X]+r\leq - n \right\rangle_\Z \punto$$

\section{Weak Néron models}
\label{sect:wnm}

\subsection{Definition and basic constructions}

\subsubsection{} In order to define the main objects involved in this manuscript
we shall need to introduce the notion of a weak Néron model and to develop some techniques  involved in the construction of such models. 

\subsubsection{} The results we are going to state in this section hold in the context of algebraic spaces, but we shall not work in such generality, thus we state them only in the context of schemes.

\begin{definition}
Let $R$ be a DVR and $K$ its fraction field and denote by $\Delta\coloneqq \Spec R$. Let $X\rightarrow \Spec K$ be a smooth morphism of schemes. We say that the scheme $\mathfrak{X}\rightarrow \Delta$ is a model for $X$ over $\Delta$ if the morphism $\mathfrak{X}\rightarrow \Delta$ is flat and there is an isomorphism $X\cong \mathfrak{X}_K$ in $\Sch_K$. 

Moreover we say that the model $\mathfrak{X}$ has a property $\mathbf{P}$ (e.g. smooth, proper) if the morphism $\mathfrak{X}\rightarrow \Delta$ has such property.
\end{definition}

\subsubsection{} The notion we are mostly interested in is that of \emph{Weak Néron Model}:

\begin{definition}
We say that a model $\mathfrak{X}\rightarrow \Delta$ of $X$ has the weak extension property if for any finite étale morphism $Z\rightarrow \Delta$, there is a bijection $\Hom_{\Delta}(Z, \mathfrak{X}) \tilde{\rightarrow} \Hom_{K}(Z_K, X)$, $(f\colon Z\rightarrow \mathfrak{X})\mapsto f|_{Z_K}$. 

A smooth model $\mathfrak{X}\rightarrow \Delta$ of $X\rightarrow \Spec K$ that satisfies the weak extension property is said to be a \emph{weak Néron model} of $X$.
\end{definition}

\begin{example}
\label{ex:semistab}
Let $X\rightarrow \Spec K$ be a smooth and proper variety and let $\pazocal{X}\rightarrow \Spec R$ be a proper regular model of $X$, then the smooth locus of $\pazocal{X}\rightarrow \Spec R$ is a weak Néron model of $X$, because the sections $\Spec R\rightarrow \pazocal{X}$ do not meet the singular locus of $\pazocal{X}_0\coloneqq X\times_{\Spec R} \Spec k$. 

Since we will mostly work with \emph{sncd} models, i.e. regular models whose central fibre $\pazocal{X}_0$ sits in $\pazocal{X}$ as a divisor with strict normal crossings, this construction will be often used throughout this manuscript.
\end{example}

\subsection{The weak Néron model of a Hilbert scheme}

\subsubsection{} Let $V\rightarrow S$ be a flat morphism of schemes and let $n\in \N_+$ a positive integer. Consider the scheme $\Sym^n(V/S)\coloneqq V^{n}/{\Sigma_n}$, where the symmetric group acts on the product by permuting the factors. There is a natural morphism 
$$ \Hilb^n(V/S) \rightarrow \Sym^n(V/S) $$
called Hilbert-Chow morphism, sending a subscheme of $V$ to a weighted sum of finite $S-$subschemes of $V$.


\subsubsection{} If $V\rightarrow S$ is a smooth map of relative dimension $2$, then $\Hilb^n(V/S)$ is a smooth scheme of relative dimension $2n$, in particular the Hilbert-Chow morphism is a resolution of singularities of $\Sym^n(V/S)$.

\subsubsection{} Our main motivation for studying the Hilbert schemes of points on a K3 surface comes from the following:

\begin{theorem}
If $X\rightarrow \Spec \C$ is a K3 surface, then $\Hilb^n(X/\C)$ is an irreducible holomorphic symplectic variety of dimension $2n$.
\end{theorem}

\subsubsection{} Also the case of an abelian surface $A\rightarrow \Spec K$ is interesting, because in that case the Hilbert scheme is anyway a smooth Calabi-Yau variety, i.e. a variety with trivial canonical bundle; moreover all IHS varieties of Kummer type over $\C$ are obtained as a deformation of a subscheme of $\Hilb^n(A/\C)$.

\subsubsection{} We shall actually construct a weak Néron model of $\Hilb^n(X)$ in \S \ref{subsect:constwnm} and in particular in Proposition \ref{prop:wNmHilbn}, but we anticipate here the strategy we are going to follow.

We start with a regular sncd model $\mathfrak{X}\rightarrow \Delta$. Its smooth locus $\mathfrak{X}_{\sm}$ is a weak Néron model of $X$ and its Hilbert schemes $\Hilb^n(\mathfrak{X}_{\sm}/\Delta)$ are smooth models of $\Hilb^{n}(X)$ but they do not satisfy the weak extension property for $n>1$. This is due to the fact that a closed subscheme $Z\subseteq X$ of length $n$ is not necessarily supported on $K-$rational points. 

On the other hand, at least for $n< \chara K$ (if it is positive), there is an integer $\tilde{n}$ such that the support of all the closed subschemes of $X$ of length $\leq n$ is defined on a suitable field extension $K(\tilde{n})/K$. It turns out that base-changing with respect to the field extension $K(\tilde{n})/K$ induces a correspondence between closed subschemes $Z\subseteq X$ of length $n$ and $\Gal(K(\tilde{n})|K)-$invariant closed subschemes of $X(\tilde{n})$ of length $n$. Via equivariant semistable reduction, which we explain in \S \ref{sect:equivsemsred}, it will be possible to construct smooth models of $X_{K(\tilde{n})}$ endowed with an action of $\Gal(K(\tilde{n})|K)$ extending its natural action on $X_{K(\tilde{n})}$. 

Thus we will be able to construct a smooth model of $\Hilb^n(X_{K(\tilde{n})})$ endowed with an action of $\Gal(K(\tilde{n})|K)$ which has a property slightly weaker than the weak extension property. Nevertheless, using what we are going to explain in this section, namely the Weil restriction of scalars, we will be finally able to construct a weak Néron model for $\Hilb^n(X)$ and use it to compute the motivic integral.

\subsection{Weil restriction of scalars}

\subsubsection{} In this paragraph we study some generalities about the functor of the restriction of scalars. The main content of this paragraph is Proposition \ref{prop:WreswNm}, which shall allow us to construct weak Néron models of a finite, tamely ramified base-change of $X$. 


%

\begin{definition}
Let $S'\rightarrow S$ be a morphism of schemes and $\pazocal{Y}\rightarrow S'$ be a scheme. The functor $\Res_{S'/S}(\pazocal{Y})\colon (\Sch_S)^{\opp}\rightarrow \Sets$ defined by $T\mapsto \pazocal{Y}(T\times_S S')$ is called the Weil restriction of scalars of $\pazocal{Y}$ along $S'\rightarrow S$. 

When $\Res_{S'/S}(\pazocal{Y})$ is represented by a scheme, we say that the Weil restriction of $\pazocal{Y}$ along $S'\rightarrow S$ exists.
\end{definition}

\subsubsection{} It follows from \cite[Theorem 7.6]{BLR} that if $S'\rightarrow S$ is a finite, flat and locally of finite presentation and $\pazocal{Y}\rightarrow S'$ is quasi-projective, then the Weil restriction of $\pazocal{Y}$ along $S'\rightarrow S$ exists. This will always be the case, throughout this manuscript.


\begin{remark}
\label{rem:shres}
Consider arbitrary morphisms of schemes $S'\rightarrow S$ and $\pazocal{X}\rightarrow S$; then the universal property of fibered products implies the following:
$$ \Res_{S'/S}(\pazocal{X}\times_S S')= \underline{\Hom}_S(S',\pazocal{X}) \virgola $$
where $\underline{\Hom}_S(S',\pazocal{X})$ is the sheaf $T\mapsto \Hom_T(S'\times_S T, \pazocal{X}\times_S T)$.
\end{remark}

%
%
%

\subsubsection{} Let $R$ be a complete DVR and let $K$ be its fraction field and $k$ be its residue field. We suppose $k$ is algebraically closed. Let $K\subseteq L$ be a finite, tame, Galois extension with $\pazocal{G}\coloneqq \Gal(L|K)$ and denote by $R_L$ the integral closure of $R$ in $L$. For an arbitrary scheme $T\rightarrow \Spec R$, the action of $\pazocal{G}$ on $\Spec R_L$ induces an action on $T_{R_L}$, thus, by precomposition, an action on the Weil restriction, $ \Res_{R_L/R}(\pazocal{X})$, of an $R_L-$scheme $\pazocal{X}$: more precisely, all $g\in \pazocal{G}$, it induce a automorphisms $g\colon T_{R_L}\rightarrow T_{R_L}$, the action of $\pazocal{G}$ sends the point corresponding to the morphism $\psi \colon T_{R_L}\rightarrow \pazocal{X}$ to the composition $g(\psi)\coloneqq \psi \circ g^{-1}$.

\begin{proposition}
\label{prop:WreswNm}
Let $\mathfrak{X}'\rightarrow R_L$ be a weak Néron model for $X_L$, then $\mathfrak{X}\coloneqq \left(\Res_{R_L/R} \mathfrak{X}'\right)^{\pazocal{G}}$ is a weak Néron model for $X$.
\end{proposition}

\begin{proof}
It follows from \cite[Proposition A.5.2]{CGP} that the operations of taking the generic fibre and taking the restriction of scalars commute, therefore 
$$(\Res_{R_L/R}\mathfrak{X}')_K=\Res_{L/K}X_L=\underline{\Hom}_K(\Spec L, X) \virgola $$
where the last equality follows from Remark \ref{rem:shres}.

Let $T\rightarrow \Spec K$ be a scheme, then a morphism $T_L\rightarrow X\times_{\Spec K} T$ is $\pazocal{G}-$invariant if and only if it factors through $T_L\rightarrow T_L/\pazocal{G}=T$, this gives a bijection between $\Hom_T(T_L, X\times_{\Spec K} T)$ and the set of sections of $X\times_{\Spec K} T\rightarrow T$, which in turn is $\Hom_K(T, X)=X(T)$.
Therefore we have that
$$ \left( (\Res_{R_L/R}\mathfrak{X'})^\pazocal{G} \right)_K\cong X \punto $$ 
%
%
%
Since $\mathfrak{X}'\rightarrow \Spec R_L$ is a weak Néron model for $X_L$, it is in particular a smooth morphism, thus, by \cite[Proposition A.5.2]{CGP} also $\Res_{R_L/R}(\mathfrak{X}')\rightarrow \Spec R$ is smooth. It follows by \cite[Proposition 3.4]{Edi} that the $\pazocal{G}-$fixed locus $\mathfrak{X}$ is smooth as well. 

In order to conclude the proof, we only need to show that all the $K-$valued points of $X$ extend to $R-$valued points of $\mathfrak{X}$, since we assumed $R$ to be complete. 

A morphism $\Spec K \rightarrow X$ induces a unique morphism, which is also $\pazocal{G}-$invariant, $\Spec L\rightarrow X_L$. Since $\mathfrak{X}'$ is a weak Néron model for $X_L$, such map extends to a unique map $\Spec R_{L}\rightarrow \mathfrak{X}'$, which is also $\pazocal{G}-$invariant and correspond, by the definition of the restriction of scalars, to a map $\Spec R_L\rightarrow \Res_{R_L/R} \mathfrak{X}'$ and, by $\pazocal{G}-$invariance it factors through a map $\Spec R_L\rightarrow\mathfrak{X}\subseteq \Res_{R_L/R} \mathfrak{X}'$.
\end{proof}

%
%

%

\subsubsection{} The following lemma says that this construction is well behaved with respect to a tower of extensions:

\begin{lemma}
\label{lem:compWres}
Let $K\subseteq F \subseteq L$ be a tower of Galois extensions such that also $K\subseteq L$ is normal. Let $G\coloneqq \Gal(L|K)$, $N\coloneqq \Gal(L/F)$ and $G/N=H \coloneqq \Gal(F|K)$. Denote by $R_F$ and $R_L$ the integral closures of $R$ in $F$ and $L$, respectively, and $\Delta_F$, $\Delta_L$ their spectra. Let $\pazocal{F}\colon \Sch_{\Delta_L}^{\opp}\rightarrow \Sets$ be a functor endowed with an action of $G$ compatible with its action on $\Delta_L$. 

Then the following two functors are naturally isomorphic:
$$ {\left(\Res_{\Delta_L/\Delta}\pazocal{F}\right)}^G \cong {\left(\Res_{\Delta_F/\Delta}{\left(\Res_{\Delta_L/\Delta_F}\pazocal{F}\right)}^N \right)}^H \punto $$
\end{lemma}

\begin{proof}
The left hand side is equal to
$$ {\left(\Res_{\Delta_F/\Delta} \left(\Res_{\Delta_L/\Delta_F}\pazocal{F}\right)\right)}^G= {\left ( {\left(\Res_{\Delta_F/\Delta} \left(\Res_{\Delta_L/\Delta_F}\pazocal{F}\right)\right)}^N \right )}^H \virgola $$
where $N$ acts as a subgroup of $G$ and $H=G/N$ inherits the action of $G$ on the $N-$invariant locus.
Thus, we only need to show that there is an $H-$equivariant isomorphism of functors:
$$ \Res_{\Delta_F/\Delta}{\left( \Res_{\Delta_L/\Delta_F}\pazocal{F}\right)}^N \cong  {\left(\Res_{\Delta_F/\Delta} \left(\Res_{\Delta_L/\Delta_F}\pazocal{F}\right)\right)}^N  \punto $$
Given a scheme morphism $T\rightarrow \Delta$, we have that
\begin{eqnarray*}
\Res_{\Delta_F/\Delta} {\left(\Res_{\Delta_L/\Delta_F}\pazocal{F}\right)}^N(T) & =  & {\left(\Res_{\Delta_L/\Delta_F}\pazocal{F}\right)}^N(T\times_\Delta \Delta_F) \\
 & = & {\left(\pazocal{F}(T\times_\Delta \Delta_L)\right)}^N \\
 & = & {\left(\Res_{\Delta_F/\Delta} \left(\Res_{\Delta_L/\Delta_F}\pazocal{F}\right)(T)\right)}^N \virgola 
\end{eqnarray*}
and we are done.
\end{proof}

\subsubsection{} The construction above can be made more explicit: in the following paragraphs we will describe the central fibre and the canonical divisor of $\mathfrak{X}$ in terms of $\mathfrak{X}'$.

\subsection{Weil restriction and the central fibre}
\label{par:centrfib}

\subsubsection{} This subsection and the next one summarize some results contained in an unpublished manuscript of Lars Halle and Johannes Nicaise. I am grateful to them for letting me use these results which are crucial for the computation of the motivic integral in \S \ref{subsect:compmotint}. We keep the notation of the previous paragraph.


\subsubsection{}
\label{par:bunmod}
The inclusion $\mathfrak{X}\subseteq \Res_{R_L/R} \mathfrak{X}'$ corresponds, according to the definition of the restriction of scalars, to a map of $R_L-$schemes
$$ h\colon \mathfrak{X}\times_R R_L\rightarrow \mathfrak{X}' \virgola $$
which gives, over the special fibres a morphism of $k-$schemes:
$$ h_k\colon \mathfrak{X}_k \rightarrow \mathfrak{X}_k' \punto $$

\subsubsection{} On the other hand, we can characterize $\mathfrak{X}_k$ in a different way, using the \emph{Greenberg schemes}.The following definition will cover the cases we shall use: 

Let $d=[L:K]$ and let $\mathfrak{m}$ be the maximal ideal of $R_L$; for $i\in \{0,\dots, d-1\}$, let $R_{L,i}\coloneqq R_L/(\mathfrak{m}^{i+1})$. For a separated, smooth morphism $\pazocal{A}\rightarrow \Spec R_L$, consider the functor
$$ \Gr_i(\pazocal{A})\coloneqq \Res_{R_{L,i}/k}(\pazocal{A}\times_{R_L} R_{L,i}) \virgola $$
which is representable by a separated, smooth scheme, as it follows from the proof of \cite[Proposition 7.6]{BLR}; this is also called the called \emph{level $i$ Greenberg scheme of $\pazocal{A}$}. 

Clearly, $\Gr_0(\pazocal{A})=\pazocal{A}_k$, while $\Gr_{d-1}(\pazocal{A})= \Res_{R_L/R}(\pazocal{A})_k$, since $d=[L:K]$ is also the ramification index of $R_L\rightarrow R$ at their closed points. 

\subsubsection{} In our case we have that $\mathfrak{X}_k=(\Gr_{d-1}(\mathfrak{X}')^\pazocal{G}$. Indeed, if $T$ is a $k-$scheme, we have that
$$ \mathfrak{X}_k(T)={ \left\{ { f\colon T\times_R R_L\rightarrow \mathfrak{X}' } \right\} }^{\pazocal{G}} = (\Gr_{d-1}(\mathfrak{X}')(T))^{\pazocal{G}}  \punto $$

\subsubsection{} \label{sect:affbdlWR} The natural truncation maps of Greenberg schemes define $\pazocal{G}-$equivariant affine bundles, in particular $\Gr_{d-1}(\mathfrak{X}')\rightarrow \Gr_0(\mathfrak{X'})=\mathfrak{X}'_k$ is a composition of affine bundles. By taking the $\pazocal{G}-$invariant loci of this map, we get a description, at least locally, of $\mathfrak{X}_k$ as an affine bundle over $(\mathfrak{X}'_k)^\pazocal{G}$, in the sense that for each connected component $C\subseteq \mathfrak{X}_k$, there is a connected component $C'\subseteq (\mathfrak{X}'_k)^\pazocal{G}$ such that $C$ is an $\A_k^r-$bundle over $C'$, where $r= \dim(\mathfrak{X}'_k)-\dim C'$. In particular the following relation holds in $K_0(\Var_k)$:
$$ [C]=\L^{\dim (\mathfrak{X}_k)-\dim C'}[C'] \punto $$

\subsection{Weil restriction and canonical divisor}

\subsubsection{} Let us keep the notation introduced in the previous paragraph, but we also assume that $X$ is a Calabi-Yau variety, i.e. it has trivial canonical bundle, and that a volume form $\omega \in \Omega_{X/K}^{\dim X}(X)$ is given. Let $\omega_L\in \Omega_{X_L/L}^{\dim X}(X_L)$ be the pull-back of $\omega$ under the base-change map. In this paragraph we shall study the order of vanishing of $\omega$ on each component of $\mathfrak{X}_k$, which we will  define as follows, adapting the definition given in \cite[\S 4.1]{LS}.

\subsubsection{} Let $p\in \mathfrak{X}_k$ be a closed point. Since $R$ is an Henselian ring and since $\mathfrak{X}\rightarrow \Spec R$ is smooth, there is at least a section $\psi \colon \Spec R \rightarrow \mathfrak{X}$ such that $\psi(0)=p$. Consider the line bundle $L\coloneqq \psi^\ast \Omega^{\dim X}_{\mathfrak X/R}$ over $\Spec R$. There is $a\in \Z$ such that $\pi^a \omega$ extends to a global section of $\omega'\in\Omega^{\dim X}_{\mathfrak X/R}(\mathfrak{X})$, where $\pi\in R$ is the uniformizer. So its pull-back $\psi^\ast (\omega')$ is a global section of $L$. Let $M\coloneqq L/\psi^{\ast}\omega' \bigO_{\Spec R}$ be the quotient of $\bigO_{\Spec R}-$modules.

\begin{definition}
The order of $\omega$ at $p$ is defined as:
$$ \ord_p(\omega)\coloneqq \inf \{b\in \N\colon \pi^b M = 0 \} -a \punto $$

If $C\subseteq \mathfrak{X}_k$ is a connected component, then $\ord_p(\omega)$ does not depend on the coice of the closed point $p\in C$, so we define $\ord_C(\omega)$ as the order of $\omega$ at any of its closed point. If $\ord_C(\omega)>0$ we say that $C$ is a \emph{zero} of $\omega$, if $\ord_C(\omega)<0$ we say that it is a \emph{pole} of $\omega$.
\end{definition}

\subsubsection{}
Let $Z$ be a smooth scheme defined over a field $F$ and let $V\rightarrow Z$ be a vector bundle over $Z$. Consider a cyclic group $G\cong \mu_d$ acting equivariantly on $V\rightarrow Z$ and let $z\in Z$ be a fixed point. There is a unique sequence of integers $(j_1,j_2,\dots, j_{\rk V})$ such that $0\leq j_1\leq j_2\leq\cdots\leq j_{\rk V} \leq d-1$ such that $V_z$ has a basis $v_1,\dots, v_{\rk V}$ of eigenvectors such that $\zeta \star v_i= \zeta^{-j_i} \cdot v_i$ (where $\zeta$ is any generator of $\mu_d$); the tuple $(j_i)_i$ is called the \emph{tuple of exponents} of the $G-$action. 

\begin{definition}
\label{def:conact}
We define the \emph{conductor} of the action of $G$ in $z$ as the sum:
$$ c(V,z)\coloneqq \sum_{i=1}^{\rk V} j_i  \punto $$
If $C\subseteq Z^G$ is an irreducible subscheme, then for all $z,z'\in C$ one has that $c(V,z)=c(V,z')$, so we simply denote by $c(V,C)$ either of the conductors. Moreover we denote $c(Z, C)$ the conductor $c(T_{Z}, C)$.
\end{definition}

\begin{lemma}
\label{lem:ordcond}
Let $C$ be a connected component of $\mathfrak{X}_k$ and let $C'=h(C)$, where $h$ is the map defined in \S \ref{par:bunmod}. Then:
$$ \ord_C(\omega)= \frac{\ord_{C'}(\omega_L)-c(\mathfrak{X}'_k,C')}{[L\colon K]} \punto $$
\end{lemma}

\begin{proof}
The map $\mathfrak{X}\times_R R_L\rightarrow \mathfrak{X}'$ induces a monomorphism
$$\alpha\colon (\psi')^\ast \Omega_{\mathfrak{X}'/R_L}\rightarrow \psi^\ast\Omega_{\mathfrak{X}/R}\otimes_R R_L$$
sending $\omega_L$ to $\omega\otimes 1$; in particular
$$\length\left({\alpha\left({(\psi')^\ast  \Omega_{\mathfrak{X}'/R_L}}\right)/ \langle \omega\otimes 1 \rangle }\right)= [L:K] \length \left({(\psi')^\ast  \Omega_{\mathfrak{X}'/R_L}/ \langle\omega_L\rangle }\right) \virgola $$
thus the statement shall follow from the fact that $\coker \alpha= c(\mathfrak{X}'_k, C')$. 

On the other hand, under the identification 
$$T_{\mathfrak{X}/R}=\underline{\Hom}_R(R[\varepsilon]/(\varepsilon^2), \mathfrak{X}) =  \underline{\Hom}_{R_L}(R_L[\varepsilon]/(\varepsilon^2), \mathfrak{X}')^{\Gal(L|K)}= T_{\mathfrak{X}'/{R_L}}^{\Gal(L|K)} \virgola $$
the tangent map $T_h\colon T_{\mathfrak{X}/R\times_R R_L}\rightarrow T_{\mathfrak{X}'/R_L}$ induces a map 
$$\beta\colon (\psi')^\ast (T_{\mathfrak{X}'/R_L})^{\Gal(L|K)}\otimes_R R_L\rightarrow (\psi')^\ast T_{\mathfrak{X}'/ R_L} \punto $$
Fix a base of eigenvectors of $T_{\mathfrak{X}_k'}$; the upcoming Lemma \ref{lem:Gmodule}, applied to the subspaces generated by each element of the base, implies that $\coker(\beta)= \oplus_{i=1}^{d} R_L/\mathfrak{m}_L^{j_i}$, hence
$$ \coker(\alpha)= \bigwedge_{i=1}^d R_L/\mathfrak{m}_L^{j_i}= R_L/\mathfrak{m}_L^{c(\mathfrak{X}, C')}  \punto $$
\end{proof}

\begin{lemma}
\label{lem:Gmodule}
Let $M$ a free $R_L-$module of rank 1. Assuming that $\Gal(L|K)$ acts $R-$linearly on $M$ from the left. Let $j$ be the exponent of the action induced on $M\otimes_{R_L} k$, as in Definition \ref{def:conact}. Then the natural morphism $M^{\Gal(L|K)}\otimes_R R_L\rightarrow M$ has cokernel isomorphic to $R_L/\mathfrak{m}_L^j$.
\end{lemma}

\begin{proof}
Let us fix an element $v\in M$ such that $0\neq v\otimes 1 \in M\otimes_{R_L} k$; by our hypothesis we have that $(\zeta\ast v)\otimes 1 = \zeta^{-j} v	\otimes 1$. Let $\pi_L$ a uniformizer for $R_L$ such that $\pi_L^d\in K$; if $0\leq b\leq d-1$ the vectors $v_b\coloneqq \pi_L^b v\otimes 1\in M\otimes_{R} k$ form a base of the vector space $M\otimes_{R} k \cong M\otimes_{R_L} R_L/\frak{m}_L^d$. Moreover that one is a base of $G-$eigenvectors, for $\zeta\ast v_b = \zeta^{b-j} v_b$. 

By Henselianity we can lift the base $\{v_b\}$ to an $R-$base $\{w_b\colon 0\leq b\leq d-1\}$ of $M$ such that $\zeta\ast w_b = \zeta^{b-j}w_b$. 

Now let $x= a_0w_0+\cdots + a_{d-1}w_{d-1}\in M$ be an arbitrary element. We have that $x\in M^G$ iff $x-\zeta\ast x=0$, i.e. iff
$$ \sum_{b=0}^{d-1} a_b(1-\zeta^{b-j})w_b =0 \virgola $$
therefore, the $R-$module $M^G$ is generated by $w_j$.

It follows that $M^G\otimes R_L$ is sent onto $w_jM\subseteq M$, which leads to our coveted statemet.
\end{proof}

\section{Equivariant semistable reduction}
\label{sect:equivsemsred}


\subsection{Preliminaries on logarithmic geometry}

\subsubsection{} As seen in Example \ref{ex:semistab}, every proper regular model of a smooth variety contains a weak Néron model, on the other hand it is not always possible to construct a semistable model over a DVR of all the varieties defined on its fraction field $K$. Given a variety $X$ over $K$ that admits a regular \emph{sncd} model, under suitable assumptions there exists a finite extension $F/K$ such that the basechange $X_F$ admits a semistable model over the ring of integers in $F$. In this section we will show a way to produce such models over Galois extensions of a given field in such a way that the Galois group acts naturally on the model.

\subsubsection{} From now on, let $R$ be a DVR and, let $K$ be its fraction field and $k$ its residue field, which we assume being perfect and containing all the roots of unity. Let $\Delta\coloneqq \Spec R$, $0\in \Delta$ be its closed point and $\Delta^\ast\coloneqq \Delta \backslash 0= \Spec K$. Due to our assumptions, given a positive integer $m$, not divisibe by the characheristic of $k$ (if positive), there is a unique finite morphism of degree $m$, $\Delta(m)\rightarrow \Delta$, which is tame and totally ramified over $0$ and such that $\Delta^\ast(m)\rightarrow \Delta^\ast$ is a Galois étale morphism with automorphism group $\Aut_\Delta(\Delta(m))\cong \mu_m$; we denote by $0_m\in \Delta(m)$ the ramified point, omitting the index $m$ if no confusion with $0\in \Delta$ may arise.

\subsubsection{} We use the standard language of monoids and logarithmic geometry, as in \cite{GR} or in \cite{Bul} and we refer to \cite{Qu} for the material concerning toric schemes. We only recall here the notions that will be used later in this section:

\subsubsection{} Let $V^\dagger=(V,\pazocal{M}_V)$ be a locally Noetherian fs logarithmic scheme and $x\in V$ be a point. We say that $V^\dagger$ is log regular at $x$ if the following:
\begin{itemize}
\item The ring $\bigO_{V,x}/\pazocal{M}_{V,x}\bigO_{V,x}$ is regular;
\item $\dim \bigO_{V,x} =\dim \bigO_{V,x}/\pazocal{M}_{V,x}\bigO_{V,x} + \dim \pazocal{M}_{V,x}$.
\end{itemize}
are satisfied. A log structure $V^\dagger$ is logarithmically regular if it is log regular at each of its points. 


\subsubsection{} A \emph{(sharp) monoidal space} is a couple $(T, \pazocal{M}_T)$ consisting of a topological space $T$ and a sheaf of (sharp) monoids over $T$. A morphism of monoidal spaces $(T', \pazocal{M}_{T'})\rightarrow (T, \pazocal{M}_T)$ is the datum of a continuous map $f\colon T'\rightarrow T$ and a map of sheaves of monoids $h\colon f^{-1}(\pazocal{M}_{T'})\rightarrow \pazocal{M}_{T}$. If $(T, \pazocal{M}_T)$ is a monoidal space, one can define its sharpification $(T, \pazocal{M}_T)^\sharp\coloneqq (T, \pazocal{M}_T^\sharp)$, where $\pazocal{M}_T^\sharp$ is the sheafification of $U\mapsto \pazocal{M}_T(U)^\sharp$.

\subsubsection{} Let $P$ be a monoid, consider the set $\Spec P$ of its prime ideals, i.e. its submonoids whose complement in $P$ is again a submonoid of $P$, endowed with the topology generated by $\{D(f)\colon f\in P\}$, where $D(f)= \{\mathfrak{p}\in \Spec P \colon f\notin \mathfrak{p}\}$. Let $\pazocal{M}_P$ be the sheaf of monoids such that $\pazocal{M}_P(D(f))=P_f$. The space $(\Spec P, \pazocal{M}_P)$, or simply $\Spec P$, is the \emph{spectrum} of $P$. Moreover $(Spec P)^\sharp$ is called the \emph{sharp spectrum} of $P$.

\subsubsection{} A \emph{fan} is a sharp monoidal space $(F,\pazocal{M}_F)$, that can be covered by open subsets isomorphic to sharp spectra of monoids. A fan is \emph{locally fine and saturated} (or locally fs) if it can be covered by spectra of fs monoids. The category of fans is a full subcategory of the category of sharp monoidal spaces, i.e. a morphism of fans is just a morphism of monoidal spaces between fans.

\subsubsection{} Let $(F,\pazocal{M}_F)$ be a fan. A subdivision is a morphism of fans $\varphi \colon (F',\pazocal{M}_{F'})\rightarrow (F,\pazocal{M}_F)$ such that:
\begin{itemize}
\item For every $t\in F'$, the map induced on stalks $\pazocal{M}_{F, \varphi(t)}^{\grp} \rightarrow \pazocal{M}_{F,t}^{\grp}$ is surjective;
\item The composition with $\varphi$ induces a bijection $\Hom(\Spec \N, F')\rightarrow \Hom(\Spec \N, F)$.
\end{itemize}

\subsubsection{} Let $V^\dagger=(V,\pazocal{M}_V)$ be a log regular scheme. There is a fan associated to $V$ whose underlying topological space is the topological subspace of $V$ whose points are $F(V)\coloneqq \{x\in V\colon \mathfrak{m}_x \mbox{ is generated by } \pazocal{M}_{V,x}\backslash \bigO_{V,x}^\times\}$, where $\mathfrak{m}_x$ is the maximal ideal of $\bigO_{V,x}$, and whose structural sheaf $\pazocal{M}_{F(V)}$ is the restriction of $\pazocal{M}_V$.

\subsection{The construction of a semistable model}

\label{subsect:tripoiloc}

\subsubsection{} Let $\Delta^\dagger$ be the log scheme supported on $\Delta$ with the divisorial log structure associated to $0\in \Delta$. Similarly $\Delta^{\dagger}(m)$ shall denote the scheme $\Delta(m)$ introduced above endowed with the divisorial log structure for $0\in \Delta(m)$. Let $S\rightarrow \Delta$ a flat morphism of relative dimension 2 with smooth generic fibre and let $S^\dagger$ be the divisorial logarithmic structure associated to the central fibre $S_k$. Assume $S^\dagger\rightarrow \Delta^\dagger$ to be logarithmically smooth. Let us fix a positive integer $N$ such that the multiplicity of any irreducible component of $S_0$ divides $N$ and let $m$ be an arbitrary positive integer; in case $\chara K=p>0$ let us assume that the multiplicities of the irreducible components of $S_0$ are coprime with $p$, choose $N,m$ coprime with $p$. Let $\Gamma$ be the fan associated to the logarithmic structure of $S^\dagger$.

\subsubsection{} The fs-basechange $S^\dagger\times_{\Delta^\dagger}^{\fs} \Delta(mN)^\dagger\rightarrow \Delta^\dagger$ which is a normal space, yet not necessarily regular, is canonically endowed with an equivariant $\mu_{mN}-$action which induces an action on the fan of $S^\dagger\times_{\Delta^\dagger}^{\fs} \Delta(mN)^\dagger$, which we call $\Gamma(mN)$.

\subsubsection{} We define an equivariant resolution of $S^\dagger\times_{\Delta^\dagger}^{\fs} \Delta(mN)^\dagger\rightarrow \Delta^\dagger$ as follows:
\begin{enumerate}[label= \emph{Step \roman*.}]
\item \emph{Choosing the locus for the blow-up.} We choose any conic cell of $\Gamma(mN)$ and we choose a simplicial subdivision of such cell; we copy this subdivision in all the cells lying in the same orbit of the action of $\mu_{mN}$. We repeat this step until we obtain a simplicial cone complex, which we name $\widetilde{\Gamma}(mN)$.
\item \emph{The blow-up gives a semistable model.} 
Because of \cite[Lemma 2.4 (2) and Lemma 2.7 (2)]{Sai}, the simplicial subdivision in step 1 induces a resolution $p\colon S(mN)^\dagger\rightarrow S^\dagger\times_{\Delta^\dagger}^{\fs} \Delta(mN)^\dagger$ whose underlying scheme $S(mN)$ is semistable.
\item \emph{Description of the action of $\mu_{mN}$.} Let $\zeta\in \mu_{mN}$ a generator and let $\zeta\colon S^\dagger\times_{\Delta^\dagger}^{\fs} \Delta(mN)^\dagger\rightarrow S^\dagger\times_{\Delta^\dagger}^{\fs} \Delta(mN)^\dagger$ its induced isomorphism. By composition we obtain a birational morphism $S(mN)\rightarrow S^\dagger\times_{\Delta^\dagger}^{\fs} \Delta(mN)^\dagger$ and, thence, a birational map $S(mN)\dashrightarrow S(mN)$.This is actually an isomorphism because the subdivision of the fan is equivariant and so is the ideal involved in the blow-up.
\end{enumerate}

\subsubsection{} We conclude the section with the following lemma which grasp the most important, at least for our purpose, property of the action of $\mu_{mN}$ on $S(mN)$; in fact this is the only reason why we did perform the construction in this way:

\begin{definition}
For each point $p\in S(mN)_k$ let $\Stab_p\subseteq \mu_{mN}$ be the subgroup consisting of the elements that fix $p$.
\end{definition}

\begin{lemma}
\label{lem:constab}
$\Stab_{-}$ is locally constant on $S(mN)_{k,\sm}$. 
In particular if a point $p\in S(mN)_{k,\sm}$ is fixed under the action of $\mu_N$, then the whole connected component of $S(mN)_{k,\sm}$ containing $p$ is fixed under such action.
\end{lemma}

\begin{proof}
Each point of $S$ admits an étale neighbourhood $U\twoheadrightarrow W \subseteq S$ (with $W$ being a Zariski open subset of $S$) such that the map $U\rightarrow \Delta$ splits through a smooth map $U\rightarrow V \subseteq \Spec R[P]$, where $P$ is a torsion-free monoid giving a local chart $P\rightarrow \pazocal{M}$ for the logarithmic structure on $U$. The embedding $V\hookrightarrow \Spec R[P]$ is given by the ideal $(\chi^v-\pi)\subseteq R[P]$, where $v\in P$ is the image of $1$ under the map of monoids $\N\rightarrow P$ giving the chart for $U^\dagger \rightarrow \Delta^\dagger$. 

Let $U(mN)\rightarrow W(mN)\subseteq S(mN)$ be the base change of $U\rightarrow W\subseteq S$ with respect to the map $S(mN)\rightarrow S$ arising from the construction above and let $V(mN)\rightarrow V\times_\Delta \Delta(mN)$ be the toric resolution arising from the same subdivision that we performed above; in particular there is a smooth map $\psi \colon U(mN)\rightarrow V(mN)$ which is equivariant with respect to the natural action of $\mu_{mN}=\Gal(K(mN)|K)$. 

Let us omit the symbol $\dagger$ from the log-schemes, all the object we deal with are interpreted in the category of log-schemes unless differently stated. Consider the following equivariant map:
$$ V(mN)\rightarrow V\times^{\fs}_{\Delta} \Delta(mN) \subset \Spec R{\left[ {\left( P\oplus_{\N} \frac{1}{mN}\N \right)}^{\sat} \right]} \eqqcolon \Spec R[P] \times^{\fs}_{\Delta} \Delta(mN) \virgola $$
where the amalgamated sum is taken with respect to the obvious inclusion $\N\subseteq \frac1{mN}\N$ and $\N\rightarrow P$, $1\mapsto v\in P$ defined above. This map is an isomorphism outside the central fibre of the two schemes.

There is an integer $d$ and a primitive element $v_1\in P$ such that $v=d\cdot v_1$ (in our case, i.e. of a log structure arising from an \emph{snc} divisor, $d$ coincides with the \emph{g.c.d.} of the multiplicities of the components containing $p$, in general it was defined as \emph{root index} in \cite{BN}); by our assumption on $N$, we have that $d|mN$, thus the following identity holds:
$$ V\times^{\fs}_{\Delta} \Delta(mN) = (V\times^{\fs}_{\Delta} \Delta(d)) \times^{\fs}_{\Delta(d)}  \Delta(mN) \punto $$

We can, thus, consider the base-changes separately; the undelying scheme of $V\times^{\fs}_{\Delta} \Delta(d)$ is the normalization of the base-change in the category of schemes \cite[Proposition 3.7.1]{BN}. We have that $V\times_\Delta \Delta(d)\cong \Spec R(d)\left[ { P} \right]/(\chi^{dv_1}-\varpi^d)$, where $\varpi\in R(d)$ is a uniformizer such that $\varpi^d=\pi$. 
The chart of $\Spec R(d)[ P ]/(\chi^{dv_1}-\varpi^d)$ is given by the map of monoids \begin{align*}
 P\oplus_{\N} \frac{1}{d}\N  &\rightarrow R(d)\left[ { P} \right]/(\chi^{dv_1}-\varpi^d) \\ 
\left(0,\frac1d\right) &\mapsto \varpi \punto
\end{align*} One checks that $\displaystyle {\left(P\oplus_{\N} \frac{1}{d}\N\right) }^{\sat} \cong P\oplus \Z/d\Z$, where the latter factor is generated by $\displaystyle \left(v_1, -\frac1d\right)$.
Then $V\times^{\fs}_{\Delta} \Delta(d)\cong \coprod_{i=0}^{d-1} \Spec R(d)\left[ {
 P} \right]/(\chi^{v_1}-\zeta_d^i\varpi)$, and $\Gal (K(d)|K)$ acts on it via a cyclic permutation of the components.

We now assume that $d=1$, i.e. that $v$ is primitive. In this case the monoid $(P\oplus_\N \frac1{mN}\N)^{\sat}$ is sharp \cite[Proposition 2.2.2 (3)]{BN} and the inclusion $P^{\grp}\subseteq (P\oplus_\N \frac1{mN}\N)^{\grp}$ induces an étale map of tori $T(mN)\coloneqq\Spec \Z\left[ (P\oplus_\N \frac1{mN}\N)^{\grp} \right]\rightarrow T\coloneqq \Spec \Z[P^{\grp}]$ of degree $mN$, which is the quotient with respect to the action of $\mu_{mN}$; in particular the group $\mu_{mN}$ acts freely and transitively on the kernel of such map and there is an exact sequence of group schemes over $\Spec \Z$:
\begin{equation}
\label{eqn:exactori}
 1\rightarrow \mu_{mN}\rightarrow T(mN)\rightarrow T \rightarrow 1 \punto
\end{equation}
On the other hand, let us consider the $1-$codimensional subtorus $T'\subseteq T(mN)$ corresponding to the quotient $P(mN)^{\grp}\rightarrow P(mN)^{\grp}/\langle u \rangle$, where $u$ is the image of the generator of $\frac{1}{mN}\N$ in $P(mN)$. 
\begin{claim}
\label{clm:subtoract}
The action of $T(mN)_R$ on $\Spec R[P(mN)]$ induces an action 
$$ T'_R\times_{\Delta} (V\times_\Delta^{\fs} \Delta(mN)) \rightarrow V\times_\Delta^{\fs} \Delta(mN) $$
making $V \times_\Delta^{\fs} \Delta(mN)\rightarrow \Delta$ a toric scheme with respect to the torus $T'$.
\end{claim}
The equivariant toric resolution $V(mN)\rightarrow V\times^{\fs}_\Delta \Delta(mN)$ is an isomorphims over an open set containing the dense toric orbit, hence $\Gal(K(mN)|K)$ acts also on $V(mN)$ as a subgroup of $T'$. 

Let $O \subseteq V(mN)$ be a locally closed stratum in the canonical stratification of $V(mN)$, i.e. an orbit for the action of $T'$, as described in \cite[2.1.13]{Qu}; then a suitable quotient $T'\twoheadrightarrow \overline{T}$ acts freely on $O$ and, thus, the image of $\mu_{mN}$ in $\overline{T}$ acts freely on $O$, hence the stabilizer of an arbitrary point $q\in O$ acts trivially on the whole orbit. 

In the general case, we may apply the above argument to the map $\Delta(mN)\rightarrow \Delta(d)$, obtaining the following chain of maps:
$$ V(mN)\rightarrow V\times^{\fs}_\Delta \Delta(mN) \rightarrow V\times^{\fs}_\Delta \Delta(d) \rightarrow  V  \virgola $$
where the first map is a map of toric schemes (over $\Delta(d)$) whose main torus fits in the sequence
$$ 1\rightarrow \mu_{\frac{mN}{d}}\rightarrow T(mN) \rightarrow T(d) \rightarrow 1 $$
and the last map is just the collapse of $d$ copies of $V$. Since the generator of $\mu_{mN}$ acts on $V(mN)$ by permuting the $d$ connected components, then the stabilizer of each point is contained in the subgroup generated by its $d-$th power, i.e. $\mu_{\frac{mN}{d}}$, in particular the fact that the stabilizer of a point is locally constant on each orbit follows from what said for the $d=1$ case.  


Let $q\in U(mN)_{k,\sm}$ be a closed point and let $O$ be the stratum of $V(mN)$ containing $\psi(q)$; and let $Q\coloneqq \psi^{-1}(O)\subseteq U(mN)$. Since $q\in U(mN)_{k,\sm}$, then $O$ is a connected component of $V(mN)_{k,\sm}$, therefore $Q$ is open in $U(mN)_{k,\sm}$. Since the map $U(mN)\rightarrow V(mN)$ is equivariant, we have that $\Stab_q \subseteq \Stab_{\psi(q)}$, thus $\Stab_q$ acts trivially on $O$. The étale map $Q\rightarrow O$ induces a $\Stab_q-$equivariant étale map of the complete local rings $\widehat{\pazocal{O}_{O,\psi(q)}}\rightarrow \widehat{\pazocal{O}_{Q,q}}$, which is an isomorphism since $k$ is algebraically closed. In particular $\Stab_q$ acts trivially on $\widehat{\pazocal{O}_{Q,q}}$ which is the formal completion of the local ring $\pazocal{O}_{Q,q}$, in particular $\Stab_q$ acts trivially in a neighbourhood of $q$. Since the fixed locus of $\Stab_q$ is also closed, it acts trivially on the whole connected component containing $q$. It follows that $\Stab_{-}$ is locally constant on $U(mN)_{k,\sm}$.

In order to conclude that the stabilizer is locally constant on $W(mN)_{k,\sm}$ as well, we shall prove that each point $q\in U(mN)_{k,\sm}$ has the same stabilizer as its image $p\in W(mN)$. Let $q_0\in U$ be the image of $q$ under $U(mN)\rightarrow U$ and similarly let $p_0\in W$ be the image of $p$ under $W(mN)\rightarrow W$. Up to replacing $U\rightarrow V$ with an open subset $U'\subseteq U\rightarrow V$ we can assume that $q_0$ is the only preimage of $p_0$ under $U\rightarrow W$; in particular the orbit of $q\in U(mN)$ is sent bijectively onto the orbit of $p\in W(mN)$ under the $\mu_{mN}-$equivariant map $U(mN)\rightarrow W(mN)$, so their stabilizers must coincide.
\end{proof}

\begin{proof}[Proof of Claim \ref{clm:subtoract}]
Consider the composition of maps of affine schemes
$$ T'_R\times_{R} (V\times_\Delta^{\fs} \Delta(mN)) \rightarrow  T_R\times_{R} \Spec R[P(mN)] \rightarrow \Spec R[P(mN)] \virgola $$
corresponding to the following composition of maps of rings
\begin{align*}
R[P(mN)] & \rightarrow & R[P^{\grp}(mN)] \otimes_{R} R[P(mN)]  &\rightarrow & R[P(mN)^{\grp}/\langle u \rangle] \otimes_{R} R[P(mN)]/(\chi^u-\pi) \\
\chi^x & \mapsto & \chi^x \otimes \chi^x & \mapsto  & \chi^{\overline{x}} \otimes \overline{\chi^{x}} \virgola
\end{align*}
where $\overline{x}\in P^{\grp}/\langle u \rangle$ denotes the projection of $x\in P(mN)$ and $\overline{\chi^{x}}$ denotes the the projection of $\chi^x\in R[P(mN)]$ into $R[P(mN)]/(\chi^u-\pi)$.
Since $\overline{u}=0$, one sees that $\chi^u-\pi$ is sent to $1\otimes \overline{\chi^u}- 1\otimes \pi=0$, thus the map factors throug
\begin{align*}
R[P(mN)]/(\chi^{u}-\pi) & \rightarrow  R[P(mN)^{\grp}/\langle u \rangle] \otimes_{R} R[P(mN)]/(\chi^u-\pi) \\
\overline{\chi^x} & \mapsto  \chi^{\overline{x}} \otimes \overline{\chi^{x}} \virgola
\end{align*}
giving an action $T'_R\times_{\Delta} V\times_\Delta^{\fs} \Delta(mN) \rightarrow V\times_\Delta^{\fs} \Delta(mN)$. 

We conclude by showing that $V\times_\Delta^{\fs} \Delta(mN)$ admit a dense orbit with respect to the action of $T'_K$. 

Let $y\in (V\times_\Delta^{\fs} \Delta(mN) )\cap T(mN)_K$ be a closed point. Since $\codim_{T(mN)_K}T'_K=1$, then $T'_K \cdot y \subseteq T(mN)_K\cdot y= T(mN)_K$ is a closed subscheme of codimension at most 1. Density of $T'_K\cdot y$ follows from the fact that $(V\times_\Delta^{\fs} \Delta(mN) )\cap T(mN)_K$ is irreducible and has codimension 1 in $T(mN)_K$.
\end{proof}

\section{The monodromy conjecture}
\label{sect:intrmonconj}

\subsection{Motivic integration}

\subsubsection{} Let us fix a complete DVR $R$ and let $K$ denote its fraction field, while $k$ its residue field, which we assume to be algebraically closed. 


\subsubsection{} Let $Y\rightarrow \Delta^\ast(m)$ be a smooth Calabi Yau variety and let $\omega$ be a volume form on $Y$. Fix a weak Néron model $\mathfrak{Y}\rightarrow \Delta(m)$ of $Y$. For a connected component $C\in \pi_0(\mathfrak{Y}_0)$ let $\ord_C(\omega)$ be the order of $\omega$, considered as a meromorphic function, on the generic point of $C$.

\subsubsection{} It follows from a result of Loeser and Sebag (see \cite[Proposition 4.3.1]{LS}), that the following definition does not depend on the choice of the weak Néron model of $Y$:

\begin{definition}(Motivic integral)
With the same notation introduced in this paragraph, we call motivic integral of the volume form $\omega$ on $Y$ the element of $\mathcal{M}_k$ given by the following sum:
$$ \int_Y \omega d\mu = \sum_{C\in \pi_0 \mathfrak{Y}_0} [C]\L^{-\ord_C(\omega)} \punto $$
\end{definition}

\subsubsection{} Now fix a Calabi Yau variety $X\rightarrow \Delta^\ast$, together with a volume form $\omega\in \omega_X(X)$. 

For every positive integer $m$, define $X(m)\coloneqq X\times_{\Delta^\ast} \Delta^\ast(m)$. Since the map $\Delta^\ast(m)\rightarrow \Delta$ is an étale map, the basechange map $X(m)\rightarrow X$ is étale as well. The pull-back of $\omega$ through that map is thus a volume form on $X(m)$, which we denote by $\omega(m)$. Using this construction, a formal series with coefficients in $\mathcal{M}_k$ is defined:

\begin{definition}[Motivic Zeta Function]
\label{def:motzeta}
Keep the notation of this paragraph. The \emph{Motivic Zeta Function} of $X$ with respect to the volume form $\omega$ is the formal series
$$ Z_{X,\omega}(T)\coloneqq \left( \sum_{\stackrel{m\geq 1}{\chara k \nmid m}} \int_{X(m)} \omega(m) d\mu \right) T^m \punto $$
\end{definition}

\subsubsection{} The motivic integral and the Zeta function have an equivariant counterpart in $\mathcal{M}_k^G[[T]]$, provided that the volume form of the Calabi-Yau variety is chosen to be $G-$equivariant.

\subsection{Rational functions in $\mathcal{M}_k[[T]]$}

\subsubsection{} The notions discussed in this paragraph are introduced, for instance, in \cite{NX}.

\begin{definition}
Let $F\in \mathcal{M}_k[[T]]$, we say that $F$ is rational if there is a finite set $S\subseteq \N\times \N_+$ such that $\displaystyle F\in \mathcal{M}_k\left[ T, \frac1{1-\L^{-a}T^b} \colon (a,b)\in S \right]$. 

In such case, we say that $F$ has a pole of order at most $n\in \N$ in $q\in \Q$ if there exist a finite set $S'\in \N\times \N_+$ such that $\displaystyle \frac{a}b=q (a,b)\notin S'$ and a positive integer $N$ such that
$$ (1-\L^{-qN}T^N)^n F \in \mathcal{M}_k\left[ T, \frac1{1-\L^{-a}T^b} \colon (a,b)\in S' \right] \punto $$
We say that $F$ has a pole of order $n\geq 1$ in $q\in \Q$ if $F$ has a pole of order at most $n$, but not a pole of order at most $n-1$.
\end{definition}


%
%

\subsubsection{} The definition can be simplified, provided that we work on a ring $\mathcal{R}$ endowed with a map $\mathcal{M}_k\rightarrow \mathcal{R}$ such that all the elements of the form $\L^r-1$, with $r \in \N\backslash \{0\}$, are invertible. The minimal such choice for $\mathcal{R}$ is, clearly, the localization in $\mathcal{M}_k$ with respect to that set of elements, i.e. $\displaystyle \mathcal{M}_k\left[ (\L^r-1)^{-1} \colon 0<r\in \N \right]$. Another natural choice for $\mathcal{R}$ is the completed Grothendieck ring of varietes: $\widehat{\mathcal{M}}_k$, where the inverse of $1-\L^r$ is $1+\L^r+\L^{2r}+\cdots$. The following lemma clarifies why in this case it is easier to define a pole:

\begin{lemma}
\label{lem:seppol}
Let $\mathcal{R}$ be a ring as above and $F\in \mathcal{R}[[T]]$ any rational function. Then $\exists N>0$ a positive integer and a finite set $S\subseteq \Q$ such that:
\begin{equation}
\label{eqn:zetaK3form}
F(T)= p(T) + \sum_{q\in S} \frac{f_q(T)}{(1-\L^{-qN}T^N)^{a_q}} \virgola
\end{equation}
for some polynomials $p, f_q\in \mathcal{R}[T]$ and positive integers $a_q$.
\end{lemma}

\begin{proof}
We begin by noticing that, given $\mu > \nu$ positive integers, one has that
$$ \frac1{(1-\L^\mu T^N)(1-\L^\nu T^N)}=\frac{1}{1-\L^{\mu-\nu}}\left(\frac{1}{1-\L^\nu T^N}-\frac{\L^{\mu-\nu}}{1-\L^\mu T^N}\right) \virgola $$
where $N\in \N$. 

By repeatedly applying this step, one obtains the following identity
$$ \frac{1}{(1-\L^\mu T^N)^a(1-\L^{\nu}T^N)^b}= \frac{p(T^N)}{(1-\L^\mu T^N)^a}+\frac{q(T^N)}{(1-\L^\nu T^N)^b} \virgola $$
for positive integers $a,b$ and polynomials $\displaystyle p,q\in \Z\left[\L,\frac{1}{1-\L^{\mu-\nu}}, t \right]$ such that $\deg_t(p)<a$, $\deg_t(q)<b$. In particular any rational function in $\mathcal{R}[[T]]$ with two poles is the sum of two functions with a single pole. We conclude the proof by induction on the number of poles of $F$.
\end{proof}

\begin{definition}
If $F$ is written as in \eqref{eqn:zetaK3form} and $q\in S$, then we say that $F$ has a pole of order at most $a_q$ in $q$. 

If, moreover, there is no integer $N'\in \N_+$ such that $f_q \in (1-\L^{-qN'}T^{N'})\mathcal{R}[T]$, then $F$ has a pole of order exactly $a_q$ in $q$. 
\end{definition}

\begin{remark}
\label{rem:strongtoweak}
Let $F\in \mathcal{M}_k[[T]]$ and let $\widetilde{F}\in \mathcal{R}[[T]]$ be the image of $F$ under the completion map (or localisation map). 

Because of Lemma \ref{lem:seppol}, any pole $q$ of $\widetilde{F}$ of order $a\in \N$ is also a pole of $F$ of order greater or equal than $a$.
\end{remark}

\subsection{Statement of the conjecture}

\subsubsection{}  We conclude this section by explaining the main problem we are going to face. 
Let $\mathcal{R}$ be one of the three rings $\mathcal{M}_k, \mathcal{M}_k\left[ (\L^r-1)^{-1} \colon 0<r\in \N \right]$ or $\widehat{\mathcal{M}}_k$. We give a statement of the monodromy conjecture that depends on how the ring of coefficients for the zeta function is interpreted:

\begin{conjecture}[Monodromy conjecture in $\mathcal{R}$]
\label{conj:monconjboth}
Let $X\rightarrow \Spec K$ be a Calabi-Yau variety and let $\omega$ be a volume form on it. Let $q\in\Q$ be a pole of $Z_{X,\omega}(T) \in \mathcal{R}[[T]]$,
then $e^{2\pi i q}$ is an eigenvalue for the monodromy operator associated to $\Gal(\overline{K}|K)$.
\end{conjecture}

\subsubsection{}  The monodromy conjecture in $\mathcal{M}_k$ implies the monodromy conjecture in \\$\mathcal{M}_k\left[ (\L^r-1)^{-1} \colon 0<r\in \N \right]$ by the Remark \ref{rem:strongtoweak}. In turn, the monodromy conjecture in \\$\mathcal{M}_k\left[ (\L^r-1)^{-1} \colon 0<r\in \N \right]$ implies the version in $\widehat{\mathcal{M}}_k$.



\section{Formal series in the motivic rings}
\label{sect:forser}

In this section, the symbol $\mathcal{R}$ shall denote one of the rings $\mathcal{M}_k, \mathcal{M}_k\left[ (\L^r-1)^{-1} \colon 0<r\in \N \right]$ or $\widehat{\mathcal{M}_k}$, unless differently specified.

We will discuss some properties of power series with coefficients in $\mathcal{R}$, then we will describe some operations that shall be useful for computing the Motivic Zeta Function in our case.

\subsection{Quotient by a group action}

\subsubsection{} Let us fix a finite group $G$, let $N \trianglelefteq G$ and let $H=G/N$ be its quotient. Consider the equivariant versions of the ring $\mathcal{R}$, which we call $\mathcal{R}^G$ and $\mathcal{R}^H$, as in \S \ref{subsect:equivring}. For an arbitrary variety $X$ endowed with an action of $G$, we consider the quotient $X/N$ which is again an algebraic space endowed with an action of the group $H$, namely the quotient action.

\subsubsection{} The \emph{example of Hironaka} shows that $X/N$ is not necessarily a scheme. In any case, there exists an open affine subscheme $\Spec A\subseteq X$ which is invariant under the action of $G$, so that $\Spec A^N\subseteq X/N$ is a scheme. We may, thus, repeat the argument for the closed $G-$invariant subscheme $X\backslash \Spec A$ and, by Noetherian induction, we stratify $X$ as a union of $G-$schemes whose quotients with respect to the action of $N$ are $H-$schemes; moreover a $G-$invariant stratification of each stratum induces an $H-$invariant stratification of its quotient. Thus there is a well defined map of groups $\mathcal{R}^G\rightarrow \mathcal{R}^H$ by $[X]\mapsto [X/N]$; this map does not preserve the products.

\subsubsection{} In the following definition we extend the map above to a map $\mathcal{R}^G[[T]]\rightarrow \mathcal{R}^H[[T]]$ and in the subsequent proposition we show that rationality of any power series is well behaved under this map.

\begin{definition}
If $\displaystyle F=\sum_{n} A_nT^n \in \mathcal{R}^G[[T]]$, we define the \emph{series of the quotients} associated to $F$ as
$$ (F/N)(T)= \sum_n (A_n/N) T^n \in \mathcal{R}^H[[T]] \punto $$
\end{definition}

\begin{proposition}
\label{prop:polquot}
For $f,g\in \mathcal{R}^G[T]$, with $g$ being of the form $\prod_{j\in J}(1-\L^{a_j}T^{b_j})$, let $F$ be the rational function $\displaystyle F(T)\coloneqq\frac{f(T)}{g(T)}$. Then $\displaystyle F/N= \frac{f/N}{g}$. 

In particular all the poles of $F/N$ belong to the set of poles of $F$.
\end{proposition}

\begin{proof}
It suffices to show the statement in the case when $f=\alpha\in \mathcal{R}^G$ is a constant.
Let $\displaystyle  \frac{1}{g(T)}= \sum_n A_nT^n$, where $A_n\in \mathcal{R}$ has a trivial $G-$action. Then
$$ \left. \left(\frac{\alpha}{g(T)}\right)\right /N=\sum_n (\alpha A_n)/N T^n = \sum_n \alpha/N A_nT^n= \frac{\alpha/N}{g(T)} \punto $$
\end{proof}

\subsection{Power structures}

\subsubsection{}  We need to define a map on the Grothendieck rings which extends the symmetric product of a variety, allowing us to talk about the symmetric product of a "difference of varieties"; in order to do so, we need to use a power structure on $K_0(\Var_k)$, thus we begin by recalling what a power structure is, as introduced in \cite{GLM}:

\begin{definition}
Let $A$ be a ring. A power structure on $A$ is a map
\begin{align*}
(1+tA[[t]])\times A & \rightarrow 1+ tA[[t]] \\
(F(t), X) & \mapsto F(t)^X
\end{align*}
satisfying the following conditions for all $F,G\in 1+tA[[t]]$ and $X,Y\in A$:
\begin{itemize}
\item $F(t)^0=1$;
\item $F(t)^1=F(t)$;
\item $(F(t)G(t))^X=F(t)^X\cdot G(t)^X$;
\item $F(t)^{X+Y}={\left(F(t)\right)}^X{\left(F(t)\right)}^Y$;
\item $F(t)^{XY}={\left(F(t)^X\right)}^Y$;
\item $(1+t)^X\in1+Xt+t^2 A[[t]]$;
\item $F(t)^X|_{t\rightarrow t^n}=F(t^n)^X$.
\end{itemize}

In fact, the last two properties are not part of the original definition, but other authors include them in their definition.
\end{definition}


\subsubsection{} In the rest of the section, for $F=\sum_n F_nT^n \in K_0(\Var_k)[[T]]$, with $F_0=1$, and for $X\in K_0(\Var_k)$ we denote by $F(T)^X$ the power structure introduced by Gusein-Zade, Luengo and Melle-Hernandez in \cite{GLM}.


\subsection{Symmetic powers}
\label{sect:studysym}

\subsubsection{} Keep the notation introduced in the previous paragraph. Let us begin with the following definition:

\begin{definition}[Symmetric power in GRV]
Let $\alpha \in K_0(\Var_k)$ and let $r\in \N$. The \emph{$r-$th symmetric power of $\alpha$} is the element $\Sym^r(\alpha) \in K_0(\Var_k)$ defined as
$$ \Sym^r(\alpha) \coloneqq [t^r] (1-t)^{-\alpha} \punto $$
\end{definition}

\subsubsection{} In particular, if $\alpha=[U]$, we have that $\Sym^r([U])=[\Sym^r(U)]$; in this sense $\Sym^{\bullet}$ exends the notion of symmetric power of an algebraic variety. In general, for $\alpha= [U]-[V]$ one gets an explicit formula by analyzing the coefficients of
$$ (1-t)^{[V]-[U]}=(1-t)^{[V]}\cdot(1+t+t^2+\cdots)^{[U]} \punto $$

\begin{example}
In order to show how this computation can be handled, we compute explicitly the coefficient of $t^2$, that is, the expression for $\Sym^2([U]-[V])$.
We know, from \cite[Theorem 1]{GLM}, that
$$ (1+t+t^2+\cdots)^{[U]}=1+[U]t+[\Sym^2(U)]t^2+o(t^2) \virgola $$
thus we get
$$ (1-t)^{[V]}=(1+[V]t+[\Sym^2(V)]t^2+o(t^2))^{-1}= 1 - [V] t +([V^2]-[\Sym^2(V)])t^2 +o(t^2)   \punto $$
It follows that
$$ \Sym^2([U]-[V])= [\Sym^2(U)]-[\Sym^2(V)]+[V]^2-[U][V]  \punto $$
\end{example}

\subsubsection{} It is possible to extend the map $\Sym^r\colon  K_0(\Var_k)\rightarrow K_0(\Var_k)$ to a map $\Sym^r\colon \mathcal{M}_k \rightarrow \mathcal{M}_k$ by $\Sym^r{\L^{-s}\alpha}\coloneqq \L^{-rs}\Sym^r(\alpha)$. In order to ensure that this map is well defined, we only need to show that, for $\alpha \in K_0(\Var_k)$ and for $s\in \N$, we have $\L^{-rs}\Sym^r(\alpha)=\L^{-r(s+1)}\Sym^r(\L\alpha)$. 

Indeed, recalling that $\Sym^r(\A^n)\cong \A^{nr}$, we get
$$(1-t)^{\L\alpha}={\left((1-t)^{-\L}\right)}^{-\alpha}= (1+\L t+\L^2t^2+\cdots)^{-\alpha}=(1-\L t)^{\alpha} \virgola $$
which implies that $\forall \alpha\in K_0(\Var_k)$, $\Sym^r(\L\alpha)=\L^r\Sym^r(\alpha)$.

\subsubsection{} The map $\Sym^r$ can also be defined at the level of $\widehat{\mathcal{M}_k}$; indeed, for $\alpha, \beta\in K_0(\Var_k)$, we have that:
$$ \forall p\in \N, \quad (1-t)^{\alpha+\L^p\beta}-(1-t^\alpha)=(1-t)^\alpha\cdot\left((1-\L^p t)^{\beta}-1\right)\in \L^p  K_0(\Var_k)[[t]] \virgola $$
thus $\Sym^r(\alpha+\L^p\beta) \equiv \Sym^r(\alpha) \pmod{\L^p}$.

\subsubsection{} For the sake of completeness, we shall also define a version of $\Sym^r$ defined over \\ $\mathcal{M}_k\left[(\L^n-1)^{-1}\colon 0<n\in \N \right]$. Since $\Sym^1$ is already defined as the identity map we can proceed inductively on $r$. Let us assume that all the maps $\Sym^i$ are defined for $1\leq i \leq r-1$.
Let us first check that for $\alpha \in  \mathcal{M}_k $, the value of 
$$ \Sym^r\left( \frac{\alpha}1 \right)\coloneqq \frac{\Sym^r(\alpha)}1 \in \mathcal{M}_k\left[(\L^n-1)^{-1}\colon 0<n\in \N \right] $$
is well defined, i.e. that if $\displaystyle \frac{\alpha}1=\frac{\beta}1$, then $\displaystyle \frac{\Sym^r(\alpha)}1=\frac{\Sym^r(\beta)}1$. If $\gamma\in \mathcal{M}_k$  is such that $(\L^n-1)\gamma=0$ for some $n\in \N$, then 
$$ \Sym^r(\L^n\gamma)= \Sym^r(\gamma+ (\L^n-1)\gamma)=\sum_{j=0}^r \Sym^{j}(\gamma)\Sym^{r-j}(0)= \Sym^r(\gamma) \virgola $$
thus $(\L^{nr}-1)\Sym^r(\gamma)=0$. By an inductive argument one proves that if $\displaystyle \frac{\gamma}1=0$, then also $\displaystyle \frac{\Sym^r(\gamma)}1=0$. It follows that, whenever $\displaystyle \frac{\alpha}1=\frac{\beta}1$, then
$$ \frac{\Sym^r(\beta)}1=\frac{\Sym^r(\alpha+(\beta-\alpha))}1=\sum_{j=1}^r \frac{\Sym^j(\alpha)}1\frac{\Sym^{r-j}(\beta-\alpha)}1=\frac{\Sym^r(\alpha)}1 \punto $$
Then we define, recursively:
$$ \Sym^r\left(\frac{\alpha}{\L^n-1}\right)\coloneqq (\L^{nr}-1)^{-1}\sum_{i=1}^r \Sym^i(\alpha)\Sym^{r-i}\left( \frac{\alpha}{\L^n-1}\right) \punto $$




\begin{example}
We show how to compute this map in a specific case. For $\alpha=\displaystyle \frac{[U]}{1-\L^n}$, we have that $\Sym^2(\alpha)=\displaystyle \frac{[\Sym^2(U)]}{1-\L^{2n}} +\frac{\L^n\cdot [U]^2}{(1-\L^n)(1-\L^{2n})}$.
\end{example}


\subsubsection{} For a power series $F(T)=\sum A_nT^n\in \mathcal{R}[[T]]$, let us consider the power series obtained by plugging each coefficient of $F$ into the above mentioned maps $\Sym^r$:
$$ \Sym^r(F)(T) \coloneqq \sum \Sym^r(A_n)T^n \punto $$
These maps have very interesting properties when the function $F$ is rational, indeed in this case we are able to control the poles of $\Sym^r$ thanks to the upcoming results:

\begin{lemma}
\label{lem:sympol}
Let $\displaystyle F=\frac{\alpha T^h}{(1-\L^{-q N}T^N)^e} \in \mathcal{R}[[T]]$. Then, for $r>0$, we have that $\Sym^r(F)$ has at most one pole of order (at most) $r(e-1)+1$ in $r q$.
\end{lemma}

\begin{proof}
Let $\displaystyle F= \sum_{m\geq 0} A_{m}T^{mN+h}$; then $\displaystyle A_m= {m+e-1 \choose e-1} \alpha \L^{-m q N}$. It follows that
$$ \Sym^r(A_m)= \sum_{\beta \dashv r} { { m+e-1 \choose e-1} \choose \beta_1,\dots, \beta_r } \alpha^{\beta_1}\cdot (\Sym^2\alpha)^{\beta_2}\cdots (\Sym^r \alpha)^{\beta_r} \L^{-rm q N} \puntovirgola $$
once $\beta$ is fixed, $\displaystyle { { m+e-1 \choose e-1} \choose \beta_1,\dots, \beta_r }$ is either $0$ $\forall m\in \Z$ (this happens only if $e=1$ and $\beta_1+\cdots+\beta_r>1$) or a polynomial in $m$ of degree $(e-1)(\beta_1+\dots+\beta_r)$. 

It follows (from the fact that $\displaystyle \frac1{(1-x)^n}=\sum_m p_n(m)x^m$, where $p_n\in \Q[t]$ is a polynomial such that $\deg p_n=n-1$) that
$$  \sum_{m\geq 0} { { m+e-1 \choose e-1} \choose \beta_1,\dots, \beta_r } \alpha^{\beta_1}\cdot (\Sym^2\alpha)^{\beta_2}\cdots (\Sym^r \alpha)^{\beta_r} \L^{-rm q N} T^{m N+h} $$
is a suitable combination with rational coefficients of 
$$ \left\{ \frac{\alpha^{\beta_1}\cdot (\Sym^2\alpha)^{\beta_2}\cdots (\Sym^r \alpha)^{\beta_r}T^h}{(1-\L^{-rq N}T^N)^j} \right\}_{j=0}^{(e-1)(\beta_1+\cdots+\beta_r)+1} \punto $$
Among all the partitions of $r$, the one which gives the highest possible order of the pole is the one maximizing $\beta_1+\cdots + \beta_r$, namely $\beta=(r,0,0,\dots)$, which gives a pole of order at most $r(e-1)+1$.
\end{proof}

\subsubsection{} For the rest of the section, we denote by $\mathcal{R}$ one of the two rings \\  $\mathcal{M}_k\left[(\L^n-1)^{-1}\colon 0<n\in \N \right]$ or $\widehat{\mathcal{M}_k}$: the proofs we shall present do not hold for functions with coefficients in $\mathcal{M}_k$.

\begin{proposition}
\label{prop:prodfun}
For $i=1,\dots, s$ let $F_i=\sum_{m\geq 0} A_m^{[i]}T^m\in \mathcal{R}[[T]]$ 
be rational functions and let $\pazocal{Q}_i\subseteq \Q$ be the set of poles of $F_i$. Let $F=\sum_{m\geq 0}A_m^{[1]}\cdots A_m^{[s]} T^m\in \mathcal{R}[[T]]$.
Then $F$ is also rational and its set of poles, $\pazocal{Q}$, is contained in $\pazocal{Q}_1+\pazocal{Q}_2+\cdots+\pazocal{Q}_s$. 

Moreover, for each $q\in \pazocal{Q}$, we have that
$$ \ord_q(F)\leq \max\left\{1 - s+ \sum_{i=1}^s \ord_{q_i}(F_i) \colon q_i\in \pazocal{Q}_i \mbox{ and } \sum q_i = q \right\} \punto $$
\end{proposition}

\begin{remark}
This statement holds, with the same proof, also if we consider functions $F_i\in \mathcal{M}_k[[T]]$, provided that each of them is sum of functions with a single pole.
\end{remark}

\begin{proof}
Let us assume for a moment that, $\forall i$, $F_i= \frac{\alpha_iT^{h_i}}{(1-\L^{-q_i N} T^N)^{e_i}}$, for some $\alpha_i\in \mathcal{R}$, $0\leq h_i< N$ integers, $q_i\in \Q$, $0<e_i\in \N$. In such case $A_{mN+h_i}^{[i]}= {m+e_i-1 \choose e_i-1}\alpha_i \L^{-mq_i N}$. Thus $F=0$ unless $h_i=h~\forall i$, while in this case we have that
$$ \prod_{i=1}^s A_{mN+h}^{[i]}=\left( \prod_{i=1}^s  { m+e_i-1 \choose e_i-1} \alpha_i \right) \L^{-mqN} \virgola $$
where $q=q_1+\cdots +q_s$.
The degree of $\prod_{i=1}^s { m+e_i-1 \choose e_i-1}$, seen as a polynomial in $m$, is $\sum e_i - s$, thus we get the desired result in this case. 

For the general case it is enough to consider $F_1=F_1'+F_1''$, where $F_1'= \sum B_mT^m$ and $F_1''=\sum C_mT^m$; then, setting $F'\coloneqq \sum B_m A_m^{[2]}\cdots A_m^{[s]}T^m$ and $F''\coloneqq \sum C_m A_m^{[2]}\cdots A_m^{[s]}T^m$, we have that $F=F'+F''$ and if our statement holds for both $F'$ and $F''$ then it holds also for $F$, in particular writing all the $F_i$ as in Equation \eqref{eqn:zetaK3form}, the proposition follows by an induction on the number of their summands.
\end{proof}

\begin{lemma}
\label{lem:polestot}
Let $F\in \mathcal{R}[[T]]$ be a rational function whose set of poles is $\pazocal{Q}\subseteq \Q$, or let $F\in \mathcal{M}_k[[T]]$ be the sum of functions with at most one pole. For all $r\in \N$, let $\Sigma^{r}\pazocal{Q}$ be the set of rational numbers that are sum of $r$ elements of $\pazocal{Q}$.  Then $\Sym^r F$ is also rational and its set of poles is contained in $\Sigma^r\pazocal{Q}$. 

Moreover, for each $q\in \pazocal{Q}$, we have that
$$ \ord_q(\Sym^r F)\leq \max\left\{1 - r+ \sum_{i=1}^s \ord_{q_i}(F) \colon q_i\in \pazocal{Q} \mbox{ and } \sum q_i = q \right\} \punto $$
\end{lemma}

\begin{proof}
If the Lemma holds for $F$ and $G$, then by Proposition \ref{prop:prodfun}, it holds also for $F+G$, since $\Sym^r(F+G)= (\Sym^r F)+(\Sym^{r-1} F) G +\cdots +(\Sym^r G)$. Thus it is enough to write $F$ as in Equation $\eqref{eqn:zetaK3form}$ and notice that for each addendum the statement coincides with Lemma \ref{lem:sympol}.
\end{proof}

\section{Hilbert schemes of points on a surface}
\label{sect:motintHilb}

\subsection{Construction of a weak Néron model}

\label{subsect:constwnm}

\subsubsection{}  Let $X\rightarrow \Spec K$ be a smooth surface with trivial canonical divisor and let $\omega\in \omega_{X/K}(X)$ be a volume form on it. Let $\mathfrak{X}\rightarrow \Delta$ be a regular model whose central fibre $\mathfrak{X}_k$ is a strict normal crossing divisor of $\mathfrak{X}$. Let us keep the notation of \S \ref{sect:equivsemsred} concerning the field extension over $K$ and the corresponding base-changes. If $\chara k=p>0$, we add the further assumptions that $k$ is separably closed and that the central fibre $\mathfrak{X}_k$ has no components with multiplicity divisible by $p$.

\subsubsection{} The aim of this section is to provide a closed formula for the zeta function of $\Hilb^n(X)$; generalizing the results developed in the previous sections, one expects to be able to understand the value of the zeta function in terms of the zeta functions of $X(i)$ for $1\leq i \leq n$.



\subsubsection{}  Let $a$ be the \emph{lcm} of the multiplicities of the irreducible components of $\mathfrak{X}_k$.
For $n\in \N$, let $\tilde{n}\coloneqq a\lcm(1,2,\dots,n)$ and let $K(\tilde{n})$ be the unique totally ramified extension of $K$ whose degree is $\tilde{n}$, so that $\Gal(K(\tilde{n})/K)=\mu_{\tilde{n}}$.

\subsubsection{} Denote by $\mathfrak{X}(m\tilde{n})$ be the semistable model of $X(m\tilde{n})$ obtained from $\mathfrak{X}$ using the construction of \S \ref{subsect:tripoiloc}. As usual we denote by $\mathfrak{X}(m\tilde{n})_{\sm}$ the smooth locus of $\mathfrak{X}(m\tilde{n})\rightarrow \Delta(m\tilde{n})$. Since $\Hilb^n(\mathfrak{X}(m\tilde{n})_{\sm}/\Delta(m\tilde{n}))\rightarrow \Delta(m\tilde{n})$ is a smooth model of $\Hilb^n(X(m\tilde{n})$, we have that
$$ \mathfrak{X}^{[n]}(m) \coloneqq {\left( \Res_{\Delta(m\tilde{n})/\Delta(m)}\Hilb^n(\mathfrak{X}(m\tilde{n})_{\sm}/\Delta(m\tilde{n}))\right)}^{\mu_{\tilde{n}}} \rightarrow \Delta(m) $$
is a smooth model of $\Hilb^{n}(X(m))$.

\begin{proposition}
\label{prop:wNmHilbn}
Assume that either $K$ is perfect or $\chara K>n$, then
$\mathfrak{X}^{[n]}(m)\rightarrow \Delta(m)$ is a weak Néron model of $\Hilb^n(X(m))$.
\end{proposition}

\begin{proof}
We assume that $a=1$, i.e. that $X$ has semistable reduction on $K$. The proof in the general case will descend from Proposition \ref{prop:WreswNm} and Lemma \ref{lem:compWres}.
We just need to show that every point $\Spec K(m) \rightarrow \Hilb^n(X(m))\subseteq \mathfrak{X}^{[n]}(m)$ extends to a morphism $\Delta(m)\rightarrow \mathfrak{X}^{[n]}(m)$. 

Consider a point $\Spec K(m) \rightarrow \Hilb^n(X(m))$ and let $Z\subseteq X(m\tilde{n})$ the ($\mu_{\tilde{n}}-$invariant) subscheme representing such point. Either the closure of $Z$ in $\mathfrak{X}(m\tilde{n})$ is contained in $\mathfrak{X}(m\tilde{n})_{\sm}$ or at least one point $P\in \supp Z$ specializes to the singular locus $\mathfrak{X}(m\tilde{n})_{k,\sing}$. In the first case $\overline{Z}$ represents a morphism $\Delta(m) \rightarrow \mathfrak{X}^{[n]}(m)$ extending the given $\Spec K(m)\rightarrow \Hilb^n(X(m))$. If $P\in \supp Z$ specializes to $\mathfrak{X}(m\tilde{n})_{k,\sing}$, then its residue field $k(P)$ contains strictly $K(m\tilde{n})$, for $\mathfrak{X}(m\tilde{n})$ is a regular model of $X(m\tilde{n})$. 

Let $Q\in X(m)$ be the image of $P$ under the map $\Spec k(P)\hookrightarrow X(m\tilde{n}) \rightarrow X(m)$. The degree $[k(Q): K(m)]$ cannot divide $\tilde{n}$, otherwise $k(P)$ would be $K(m\tilde{n})$, thus $[k(Q)\colon K(m)]\geq n+1$. 

Since $Z$ is $\mu_{\tilde{n}}-$invariant, then it contains the whole orbit of $P$ which is the reduced scheme associated to the preimage of $Q$ under the map $\pi\colon X(m\tilde{n})\rightarrow X(m\tilde{n})/\mu_{\tilde{n}}\cong X(m)$. We have that $\pi^{-1}(Q)=\Spec ( k(Q)\otimes_{K(m)} K(m\tilde{n}))$, which is a reduced $K(m\tilde{n})-$algebra of dimension $[k(Q):K(m)]>n$. 

This contradicts the fact that $Z$ is a subscheme of length $n$, thus this case cannot occur and we are done.
%
\end{proof}

\subsubsection{} Our goal is to study the motivic integral of $\Hilb^n(X(m))$ using the models $\mathfrak{X}^{[n]}(m)$ constructed above. In order to simplify notation, we perform the following computations only for $m=1$, working with $\mathfrak{X}^{[n]}=\mathfrak{X}^{[n]}(1)$; similar arguments apply when $m>1$.
Even though $\Hilb^n(\mathfrak{X}(\tilde{n})_{\sm}/\Delta(\tilde{n}))$ is not a weak Néron model of $X(\tilde{n})$, it is a smooth model and the results of \S \ref{par:centrfib} apply. It follows that the connected components of $\mathfrak{X}^{[n]}_k$ are in bijection with the connected components of $\Hilb^n(\mathfrak{X}(\tilde{n})_{k,\sm})^{\mu_{\tilde{n}}}$, moreover if $C\subseteq \mathfrak{X}^{[n]}_k$ is sent to $C'$ via this bijection, then $[C]= \L^{2n-\dim C'}[C']$. 

\subsubsection{} For any $d\in \N$ dividing $\tilde{n}$ we denote by $Y(d)\subseteq \mathfrak{X}(\tilde{n})_{k,\sm}$ the subscheme consisting of the points whose stabilizer is exacty $\mu_{\nicefrac{\tilde{n}}d}$. Then $Y(d)$ is $\mu_{\tilde{n}}-$invariant, since $\mu_{\tilde{n}}$ is an abelian group and points in the same orbit have the same the stabilizer.
Because of Lemma \ref{lem:constab}, $Y(d)$ is at the same time an open and closed subscheme of $\mathfrak{X}_{k,\sm}$ and we have that:
$$ \mathfrak{X}(\tilde{n})_{k,\sm}^{\mu_{\tilde{n}/d}} = \bigsqcup_{d'|d} Y(d')  \punto $$

\begin{remark}
\label{rem:identwNm}
Since $\mathfrak{X}(\tilde{n})_{k,\sm}^{\mu_{\tilde{n}/d}}$ is a scheme of pure dimension 2, it can be identified with the central fibre of $\Res_{\Delta(\tilde{n})/\Delta(d)}{(\mathfrak{X}(\tilde{n})_{\sm})}^{\mu_{\tilde{n}/d}}$, which is a weak Néron model of $X(d)$, via the map $h_k$ described in \S \ref{par:bunmod}. Let $C\subseteq \Res_{\Delta(\tilde{n})/\Delta(d)}{(\mathfrak{X}(\tilde{n})_{\sm})}^{\mu_{\tilde{n}/d}}$ be a connected component. It follows from Lemma \ref{lem:ordcond} and from the fact that $\Gal(K(\tilde{n})|K(d))$ acts trivially on $T_{h_k(C)}$ that, , 
$$ \ord_{h_k(C)}(\omega(\tilde{n}))= \frac{\tilde{n}}d \ord_{C}(\omega(d))  \punto $$
\end{remark}



\begin{definition}
Let $n\geq 0$ be an arbitrary natural number and $\alpha = (\alpha_1,\alpha_2,\dots)\in \N^{\N_{>0}}$. We say that $\alpha$ is a \emph{partition} of $n$ if $\displaystyle \sum_{i>0}i\alpha_i = n$. In this case we write $\alpha \dashv n$.
\end{definition}

\subsubsection{} The following statement gives a decomposition of the central fibre as a union of closed and open subschemes; it is an \emph{ad hoc} partition that is more suitable, for our computation, than the "canonical" stratification of the Hilbert schemes:

\begin{proposition}
\label{prop:partstrat}
$\Hilb^n(\mathfrak{X}(\tilde{n})_{k,\sm})^{\mu_{\tilde{n}}}$ admits the following decomposition as disconnected union of subschemes:
$$ \Hilb^n(\mathfrak{X}(\tilde{n})_{k,\sm})^{\mu_{\tilde{n}}}\cong \bigsqcup_{\alpha \dashv n} \prod_{j=1}^n \Hilb^{\alpha_j}(Y(j)/\mu_j) \punto $$
Moreover, for a fixed partition $\alpha \dashv n$, the isomorphism above restricts to an isomorphim between $\prod_{j=1}^n \Hilb^{\alpha_j}(Y(j)/\mu_j)$ and a finite union of connected components of $\Hilb^n(\mathfrak{X}(\tilde{n})_{k,\sm})^{\mu_{\tilde{n}}}$.
\end{proposition}

\begin{proof}
Fix a $k-$scheme $S$. For any closed subscheme $Z\subseteq \mathfrak{X}(\tilde{n})_S$ endowed with a finite map $Z\rightarrow S$, set $Z_j\coloneqq Z \cap Y(j)_S$. If $Z$ is stable under the action of $\mu_{\tilde{n}}$, then every $Z_j$, which is the intersection of two stable schemes, is stable as well; moreover the induced action of $\mu_j= \mu_{\tilde{n}}/\mu_{\nicefrac{\tilde{n}}j}$ on $Y(j)_S$ is free, thus the induced maps $\pi_j\colon Y(j)_S\rightarrow  Y(j)_S/\mu_{j}$ and $Z_j\rightarrow \pi_j (Z_j)$ are étale of degree $j$. 

In this way, from any invariant finite $S-$subscheme of $\mathfrak{X}(\tilde{n})_S$ we construct a finite subscheme in each $Y(j)_S/\mu_{j}$; on the other hand given a sequence of finite $S-$subschemes of $Y(j)_S/\mu_{j}$ of length $\alpha_j$ we get a unique $\mu_{\tilde{n}}-$stable $S-$subscheme of $\mathfrak{X}(\tilde{n})_S$ whose length is $\displaystyle \sum_j j\alpha_j$. 

The last statement follows directly from the fact that the $Y(j)-$s are themselves open and closed subschemes of $\mathfrak{X}(\tilde{n})_{k,\sm}$.
\end{proof}

\subsubsection{} Thus, recalling that $\Hilb^{\alpha_j}(Y(j)/\mu_j)$ is pure of dimension $2\alpha_j$, we conclude that $\displaystyle \mathfrak{X}^{[n]}_k= \bigsqcup_{\alpha \dashv n}\mathfrak{X}^{[n]}_{k,\alpha}$, where $\mathfrak{X}^{[n]}_{k,\alpha}$ is an affine bundle of rank $\displaystyle \sum_j 2(j-1)\alpha_j$ on
$$ \Hilb^{\alpha_1}(Y(1))\times \dots \times \Hilb^{\alpha_n}(Y(n)/\mu_n) \punto $$
We thus have the following:

\begin{corollary}
The following equation holds in the Grothendieck ring of varieties:
$$ \left[ \mathfrak{X}^{[n]}_k \right]= \sum_{\alpha \dashv n} \prod_{j=1}^n \L^{2(j-1)\alpha_j}[\Hilb^{\alpha_j}(Y(j)/\mu_j)] \punto $$
\end{corollary}

\subsection{The volume form on $\mathfrak{X}^{[n]}$}

\subsubsection{} There is a volume form, $\omega^{[n]}$, on $\Hilb^n(X)$ that naturally arises from the given $\omega\in \omega_{X/K}$. In this paragraph we shall recall its construction and compute its zeroes and poles on $\mathfrak{X}^{[n]}$. 

\subsubsection{} Let $\pr_i \colon X^n\rightarrow X$, for $i\in \{1,\dots, n\}$, denote the projections on the factors. Then $\pr_1^\ast \omega \wedge \cdots \wedge \pr_n^\ast \omega$ is a global section of $\omega_{X^n/K}$ which, being invariant under the permutation of coordinates, descends to a global section of $\omega_{\Sym^n X/K}$, which we denote by $\varphi$. Let finally $\omega^{[n]}$ be the pull-back of $\varphi$ through the Hilbert-Chow morphism, thus $\omega^{[n]}\in H^0(\Hilb^n(X), \omega_{\Hilb^n(X)})$ is a volume form on $\Hilb^n(X)$.




\subsubsection{} Now we shall compute the zeroes and poles of $\omega^{[n]}$ seen as a rational section of $\omega_{\mathfrak{X}^{[n]}/\Delta}$. Since it is a volume form on the generic fibre of $\mathfrak{X}^{[n]}$, its zeroes or poles are all irreducible components of the central fibre. Let us fix a connected component $C\subseteq \mathfrak{X}_k^{[n]}$ and let us denote by $C'$ the connected component of $\Hilb^n(\mathfrak{X}(\tilde{n})_{k,\sm})^{\mu_{\tilde{n}}}$ such that $C\rightarrow C'$ is the affine bundle described in \S \ref{sect:affbdlWR}. In the following lemma we compute the conductor of the action of $\mu_{\tilde{n}}$ at points of $C'$ in terms of the partition of $n$ corresponding to the stratum of $\Hilb^n(\mathfrak{X}(\tilde{n})/\Delta(\tilde{n}))^{\mu_{\tilde{n}}}$ containing $C'$.

\begin{lemma}
Consider the decomposition of $\Hilb^{n}(\mathfrak{X}(\tilde{n})_{k, \sm})^{\mu_{\tilde{n}}}$ of Proposition \ref{prop:partstrat} and fix a point $[Z]$ lying inside the stratum corresponding to $\alpha \dashv n$. Then:
\begin{enumerate}
\item The conductor of the action of $\mu_{\tilde{n}}$ at $[Z]$ is 
$$c\left(\Hilb^{n}(\mathfrak{X}(\tilde{n})_{k, \sm}),[Z]\right)= \tilde{n}\sum_{j=1}^n (j-1)\alpha_j \punto $$
\item If we denote by $[Z_j]$ the point of $\Hilb^{\alpha_j}(Y(j)/\mu_j)$ corresponding to $Z_j/\mu_j$, then one has that:
$$ \ord_{[Z]}(\omega^{[n]}(\tilde{n}))= \tilde{n}\sum_{j=1}^n \ord_{[Z_j/\mu_j]}(\omega^{[\alpha_j]}(j)) \virgola $$
where $Y(j)$ are considered as part of the central fibre of a weak Néron model for $X(j)$ as in Remark \ref{rem:identwNm}.
\end{enumerate}
\end{lemma}

\begin{proof}
Since the values of $c \left(\Hilb^{n}(\mathfrak{X}(\tilde{n})_{k, \sm}),[Z] \right)$ and $\ord_{[Z]}(\omega^{[n]}(\tilde{n}))$ depend only on the connected component containing $[Z]$ and since the generic point of each connected component corresponds to a reduced scheme, we may compute them with the additional assumption that $Z$ is a reduced subscheme of $\mathfrak{X}(\tilde{n})_{k,\sm}$. In particular $Z$ is the disjoint union of $\alpha_1$ orbits of length 1, $\alpha_2$ orbits of length 2 and so on. There is an equivariant isomorphism 
$$ T_{[Z]} \Hilb^{n}(\mathfrak{X}(\tilde{n})_{k,\sm}) \cong \bigoplus_{p\in \supp Z}T_p\mathfrak{X}(\tilde{n})_{k,\sm} \punto $$ 
Let $\zeta=\zeta_{\tilde{n}}$ be a primitive root of  unity and let $\sigma$ be the unique generator of $\mu_{\tilde{n}}$ such that $\sigma$ acts on $R(\tilde{n})$ by multiplying the uniformizing parametre by $\zeta$. Consider an orbit of points $p_0,\dots, p_{j-1}\in Z$, let $e_1,e_2$ be two generators of $T_{p_0}\mathfrak{X}(\tilde{n})_k$, so that $\sigma^l(e_1), \sigma^l(e_2)$ will give a basis of $T_{p_l}\mathfrak{X}(\tilde{n})_k$ for each $l=0,\dots, j-1$. Notice that $\sigma^j(e_h)=e_h$ for $h=1,2$ since $\mu_{\nicefrac{\tilde{n}}j}$ acts trivially on the whole connected component cointaining $p_0$. 

For $i=0,\dots, j-1$, $h=1,2$ we have that
$$ (e_h,\zeta^{\nicefrac{i\tilde{n}}j}\sigma(e_h),\zeta^{\nicefrac{2i\tilde{n}}j}\sigma^2(e_h),\dots, \zeta^{\nicefrac{(j-1)i\tilde{n}}j}\sigma^{j-1}(e_h)) \in T_{p_0}\mathfrak{X}(\tilde{n})_k\oplus \cdots \oplus T_{p_{j-1}}\mathfrak{X}(\tilde{n}) $$
is an eigenvector with eigenvalue $\zeta^{-\nicefrac{i\tilde{n}}j}$. In total there are $2j$ of such eigenvectors, which constitute a basis for $T_{p_0}\mathfrak{X}(\tilde{n})\oplus \cdots \oplus T_{p_{j-1}}\mathfrak{X}(\tilde{n})_k$. 

The sum of the exponents of this base is
$$ 2\sum_{i=0}^{j-1} -\frac{i\tilde{n}}j=
-(j-1)\tilde{n} $$
We construct eigenvectors of $T_{[Z]} \Hilb^{n}(\mathfrak{X}(\tilde{n})_k)$ by putting a vector such as the above one at the coordinates corresponding to an orbit and 0 at the other coordinates. Running through all the possible orbits, we get a base of eigenvectors of $T_{[Z]} \Hilb^{n}(\mathfrak{X}(\tilde{n})/\Delta(\tilde{n}))$. Thus summing the exponents among all the eigenvectors will lead to the desired result for the conductor. 

Concerning the order of the volume form, we have that
\begin{align*}
\ord_{[Z]}(\omega^{[n]}(\tilde{n})) &=\sum_{p\in \supp Z} \ord_{p}(\omega(\tilde{n})) \\ 
& =\sum_{j=1}^n \frac{\tilde{n}}j\sum_{p\in \supp Z_j} \ord_{p}(\omega(j)) \\
& =\sum_{j=1}^n \tilde{n} \sum_{p\in Z_j/\mu_j} \ord_{p}(\omega(j)) \\
& =\tilde{n} \sum_{j=1}^n \ord_{Z_j/\mu_j}(\omega(j)^{[\alpha_j]}) \punto
\end{align*}
Where  the first and  the last equality follow from the fact that the stalk of the canonical bundle at a point with reduced support of an Hilbert scheme are the tensor products of the stalks of the canonical bundle of the surface at every point in the support.
The second equality follows from Remark \ref{rem:identwNm}.
The third equality follow from the fact that $\omega(j)$ has the same order on all the $j$ points of an orbit of $\mu_j$.
\end{proof}

\subsubsection{} We are able, now, to compute the order of $\omega^{[n]}$ at any point $z\in \mathfrak{X}^{[n]}_k$:

\begin{corollary}
Let $\pi\colon \mathfrak{X}^{[n]}_k\rightarrow \Hilb^{n}(\mathfrak{X}(\tilde{n})_{k,\sm})$ and assume that 
\begin{align*}
\pi(z) & = (\pi^1(z),\pi^2(z),\dots,\pi^n(z))  \\
 & \in \Hilb^{\alpha_1}(Y(1)) \times \Hilb^{\alpha_2}(Y(2)/\mu_2)\times \cdots \times \Hilb^n(Y(n)/\mu_n) \punto
\end{align*}

We have that $\displaystyle \ord_{z}(\omega^{[n]})= \sum_{j=1}^n \left( (j-1)\alpha_j+\ord_{\pi^j(z)}(\omega(j)^{[\alpha_j]}) \right)$.
\end{corollary}

\begin{proof}
As a direct consequence of Lemma \ref{lem:ordcond} and of the previous lemma we get:
$$\displaystyle \ord_{z}(\omega^{[n]})=\frac{\ord_{\pi(z)}(\omega(\tilde{n})^{[n]})-c\left(\Hilb^{n}(\mathfrak{X}(\tilde{n})_{k, \sm}),[Z] \right)}{\tilde{n}}= \sum_{j=1}^n \left( -(j-1)\alpha_j+\ord_{\pi^j(z)}(\omega(j)^{[\alpha_j]}) \right) \punto $$
\end{proof}

\subsection{Motivic integral}

\label{subsect:compmotint}

\subsubsection{} We keep the convention on $\mathcal{R}$ being one of the three rings $\mathcal{M}_k$, \\$\mathcal{M}_k\left[(\L^r-1)^{-1}\colon 0<r\in \N \right]$ or $\widehat{\mathcal{M}_k}$. We are now ready to perform the main computation of the manuscript; by using the models we constructed above we are able to compute a generating function for the motivic integrals of all the Hilbert schemes of points of a surface with trivial canonical bundle. More precisely the formula we are going to prove is the content of the following proposition:

\begin{theorem}
\label{thm:implicitZ}
The following identity holds true in $\mathcal{R}[[q]]$ if $\chara k=0$, while it holds true in $\mathcal{R}[q]/(q^p)$ if $\chara k=p>0$:
$$ \sum_{n\geq 0} \left( \int_{\Hilb^n(X)} \omega^{[n]} q^n \right)  =  \prod_{m\geq 1} \left( \left(  1-\L^{m-1}q^m\right)^{-\left( \int_{X(m)}  \omega(m) \right)/\mu_m} \right)   \punto $$
\end{theorem}

\begin{corollary}
\label{cor:explint}
Assume that either $\chara k=0$ or $\chara k> n$, then the following equation holds: 
$$ \int_{\Hilb^n(X)} \omega^{[n]} = \sum_{\alpha \dashv n} \prod_{j=1}^\infty  \left(  \L^{(j-1)\alpha_j} \Sym^{\alpha_j}\left( \left( \int_{X(j)}  \omega(j) \right)/\mu_j \right) \right) \punto$$
\end{corollary}

\begin{proof}
Since 
$$ \left(  1-\L^{m-1}q^m\right)^{-\left( \int_{X(m)}  \omega(m) \right)/\mu_m}= \sum_{l=0}^\infty \L^{(m-1)l}\Sym^{l}\left( \left( \int_{X(m)}  \omega(m) \right)/\mu_m \right) q^{ml} \virgola $$
it follows that, for $\alpha=(\alpha_1,\alpha_2, \dots)$ sequence such that $\alpha_j=0$ for $j\gg 1$, one has that
$$ \deg_q\left(\prod_{j=1}^\infty  \left(  \L^{(j-1)\alpha_j} \Sym^{\alpha_j}\left( \left( \int_{X(j)}  \omega(j) \right)/\mu_j \right) q^{j\alpha_i} \right)\right) = \sum_{j=1}^\infty j\alpha_j \punto $$ 
We get the desired result after identifying the coefficients of $q^n$ from Theorem \ref{thm:implicitZ}.
\end{proof}

\begin{proof}[Proof of Theorem \ref{thm:implicitZ}] We begin our computation using some identities we proved in the previous section.

\begin{align}
\sum_{n\geq 0} \left( \int_{\Hilb^n(X)} \omega^{[n]} q^n \right)  = & \sum_{n\geq 1} \left( \int_{\mathfrak{X}^{[n]}_k} \L^{-\ord_{z}(\omega^{[n]})} dz  q^n \right)  \nonumber \\
= & \sum_{n\geq 1} \left( \sum_{\alpha \dashv n} \left( \int_{\mathfrak{X}^{[n]}_{k, \alpha}} \L^{-\ord_{z}(\omega^{[n]})} dz \right)  q^n \right)  \nonumber \\
= & \sum_{n\geq 1} \left( \sum_{\alpha \dashv n} \left( \int_{\mathfrak{X}^{[n]}_{k, \alpha}}  \prod_{j\geq 1} \left( \L^{(j-1)\alpha_j -\ord_{\pi^j(z)}(\omega^{[\alpha_j]}(j))} q^{j\alpha_j} \right) dz \right) \right)   \nonumber \\
= & \sum_{\alpha\in \N^{\oplus \N_{\geq 1}}} \left( \prod_{j\geq 1}  \left( \L^{(j-1)\alpha_j} q^{j\alpha_j} \int_{\Hilb^{\alpha_j}(Y(j)/\mu_j)}  \L^{-\ord_{z}(\omega(j)^{[\alpha_j]})} dz   \right) \right)   \nonumber \\
= &  \prod_{j\geq 1}  \left( \sum_{\alpha_j\geq 0} \left( \L^{(j-1)\alpha_j} q^{j\alpha_j}\int_{\Hilb^{\alpha_j}(Y(j)/\mu_j)}  \L^{-\ord_{z}(\omega(j)^{[\alpha_j]})} dz   \right) \right) \punto     \nonumber 
\end{align}

\begin{lemma}
Let $Y\rightarrow \Spec k$ be a smooth surface endowed with a locally constant function $\nu\colon |Y|\rightarrow \mathcal{R}$. Suppose that functions $\nu^{[n]}\colon |\Hilb^n(Y)|\rightarrow \mathcal{R}$ and $\nu'^{[n]}\colon |\Sym^n(Y)|\rightarrow \mathcal{R}$ are defined in such a way that, for a given subscheme $Z\subseteq Y$ of length $n$ we have $\nu^{[n]}(Z)=\prod_{p\in \supp(Z)}\nu(p)^{\length(\mathcal{O}_{Z,p})}$ and $\nu^{[n]}=\nu'^{[n]}\circ \pi_n$, where $\pi_n\colon \Hilb^n(Y)\rightarrow \Sym^n(Y)$ is the Hilbert-Chow morphism.
Then, for an arbitrary natural number $\alpha\in\N$, the following identity holds:
$$ \int_{\Hilb^{\alpha}(Y)}  \nu^{[\alpha]}(z) dz= \sum_{\beta \dashv \alpha} \left( \prod_{l\geq 1}\left( \L^{(l-1)\beta_l}\int_{\Sym^{\beta_l}(Y)}  \nu'^{[l\beta_l]}(z) dz  \right) \right) \punto $$ 
\end{lemma}

\begin{proof}
We first suppose that $Y$ is connected and, thus, $\nu\equiv \lambda\in \mathcal{R}$ is constant. Thus we simply have that 
$$ \int_{\Hilb^{\alpha}(Y)}  \nu^{[\alpha]}(z) dz= \lambda^\alpha [\Hilb^\alpha(Y)] \punto $$
It follows from a well known result in DT theory, for instance \cite[\S 2.2.3]{Ric}, that $$ [\Hilb^\alpha(Y)]= \sum_{\beta \dashv \alpha} \left( \prod_{l\geq 1}\left( \L^{(l-1)\beta_l}[\Sym^{\beta_l}(Y)] \right) \right) \virgola $$
thus we deduce the desired statement, at least when $Y$ is connected. 

Now suppose that $C\subseteq Y$ is a connected component and that the statement holds for $Y\backslash C$. 

Using the fact that $\Hilb^\alpha(Y) = \bigsqcup_{j=0}^{\alpha}\left( \Hilb^{\alpha-j}(Y\backslash C)\times \Hilb^j(C)\right)$, we deduce that 
\begin{align*}
\int_{\Hilb^{\alpha}(Y)}  \nu^{[\alpha]}(z) dz = &  \sum_{j=0}^{\alpha} \left( \int_{\Hilb^{\alpha-j}(Y\backslash C)} \nu^{[\alpha-j]}(z) dz\right)  \cdot \left(\int_{\Hilb^j(C)} \nu^{[j]} dz\right) \\
= &  \sum_{j=0}^{\alpha} \left( \sum_{\delta^j \dashv \alpha-j} \left( \prod_{l\geq 1}\left( \L^{(l-1)\delta^j_l}\int_{\Sym^{\delta^j_l}(Y\backslash C)}  \nu'^{l[\delta^j_l]}(z) dz  \right) \right) \cdot  \right. \\
& \cdot \left.   \sum_{\gamma^j \dashv j} \left( \prod_{l\geq 1}\left( \L^{(l-1)\gamma^j_l}\int_{\Sym^{\gamma^j_l}(C)}  \nu'^{l[\gamma^j_l]}(z) dz  \right) \right) \right) \\
= &  \sum_{j=0}^{\alpha} \left( \sum_{\substack{{\delta^j \dashv \alpha-j}\\{\gamma^j \dashv j}}} \left( \prod_{l\geq 1}\left( \L^{(l-1)(\delta^j_l+\gamma^j_l)}\int_{\Sym^{\beta^j_l}(Y\backslash C)\times \Sym^{\gamma^j_l}(C)}  \nu'^{l[\delta^j_l+\gamma^j_l]}(z) dz  \right) \right)  \right) \\
= &  \sum_{\beta \dashv \alpha} \left( \prod_{l\geq 1}\left( \L^{(l-1)\beta_l}\sum_{i=0}^{\beta_l} \int_{\Sym^{\beta_l-i}(Y\backslash C)\times \Sym^{i}(C)}  \nu'^{l[\beta_l]}(z) dz  \right) \right)  \\
= &  \sum_{\beta \dashv \alpha} \left( \prod_{l\geq 1}\left( \L^{(l-1)\beta_l} \int_{\Sym^{\beta_l}(Y)}  \nu'^{l[\beta_l]}(z) dz  \right) \right)  \virgola
\end{align*}
which concludes the proof.
\end{proof}

 We plug the lemma above in our chain of equalities using $Y=Y(j)/\mu_j$, $\nu\coloneqq \L^{-\ord_z(\omega(j))}$, recalling also that $l\ord_{z}(\omega(j))=\ord_z(\omega(jl))$ and that $Y(j)$ naturally embeds in the central fibre of some weak Néron model of $X(lj)$ (as in Remark \ref{rem:identwNm} with $d=\nicefrac{\tilde{n}}{lj}$), we get:

\begin{align*}
\label{eqn:formulan}
\sum_{n\geq 0} & \left( \int_{\Hilb^n(X)} \omega^{[n]} q^n \right) =  \\
 &  = \prod_{j\geq 1}  \left( \sum_{\alpha_j\geq 0} \left( \L^{(j-1)\alpha_j} \sum_{\beta^j \dashv \alpha_j} \left( q^{j\alpha_j} \prod_{l\geq 1}\left( \L^{(l-1)\beta^j_l}\int_{\Sym^{\beta^j_l}(Y(j)/\mu_j)}  \L^{-\ord_{z}(\omega(lj)^{[\beta_l^j]})} dz  \right) \right) \right) \right)     \nonumber \\
%
%
%
%
%
%
 & =   \prod_{j\geq 1}   \left( \sum_{\beta^j \in \N^{\oplus \N_{\geq 1}}} \left( \prod_{l\geq 1}\left( \L^{(jl-1)\beta_l^j} q^{jl\beta_l^j}\int_{\Sym^{\beta^j_l}(Y(j)/\mu_j)}  \L^{-\ord_{z}(\omega(lj)^{[\beta_l^j]})} dz  \right) \right) \right)      \nonumber  \\
%
%
%
%
%
 & =  \prod_{j,l\geq 1}  \left(  \sum_{\beta^j_l\geq 0} \left(\L^{(jl-1)\beta_l^j} q^{jl\beta_l^j}\int_{\Sym^{\beta^j_l}(Y(j)/\mu_j)}  \L^{-\ord_{z}(\omega(lj)^{[\beta_l^j]})} dz  \right) \right)     \punto   \\
\end{align*}

 Recalling that, in the sum above, there is only a finite number of nonvanishing coefficients of $q^n$, for every positive integer $n$, we are allowed to group such summands in a different order; since the map
\begin{align*}
\N_+\times \N_+ & \rightarrow \N_+\times \N_+ \\
(j,l) & \mapsto (j\cdot l, j)
\end{align*}
is injective and its image is $\{(m,j)\colon j|m\}$, after the substitution $\lambda^m_j\coloneqq \beta^j_l$, we get the equivalent expression:
\begin{align}
\sum_{n\geq 0} & \left( \int_{\Hilb^n(X)} \omega^{[n]} q^n \right)  =   \nonumber \\
& = \prod_{m\geq 1}\left( \prod_{j|m}  \left(  \sum_{\lambda^m_j\geq 0} \left(\L^{(m-1)\lambda_j^m} q^{m\lambda_j^m}\int_{\Sym^{\lambda_j^m}(Y(j)/\mu_j)}  \L^{-\ord_{z}(\omega(m)^{[\lambda_j^m]})} dz  \right) \right) \right)      \nonumber  \\
& =  \prod_{m\geq 1}  \left(  \sum_{\lambda^m\in \N^{\Div(m)}} \left( \left(\L^{(m-1)}q^m\right)^{\sum_{j|m}\lambda_j^m} \prod_{j|m} \left(\int_{\Sym^{\lambda_j^m}(Y(j)/\mu_j)}  \L^{-\ord_{z}(\omega(m)^{[\lambda_j^m]})} dz  \right) \right) \right)      \nonumber  \\
&  =  \prod_{m\geq 1}  \left(  \sum_{r_m\geq 0} \left( \left(\L^{(m-1)}q^m\right)^{r_m} \int_{\Sym^{r_m}\left(\left(\sqcup_{j|m} Y(j)\right)/\mu_m \right)}  \L^{-\ord_{z}(\omega(m)^{[r_m]})} dz  \right) \right)         \nonumber  \\
&  =  \prod_{m\geq 1}  \left(  \sum_{r_m\geq 0} \left( \left(\L^{(m-1)}q^m\right)^{r_m} \Sym^{r_m}\left(\int_{\left(\sqcup_{j|m} Y(j)\right)/\mu_m}  \L^{-\ord_{z}(\omega(m))} dz\right)  \right) \right)         \nonumber      \\
&  =  \prod_{m\geq 1}  \left(  \sum_{r_m\geq 0} \left( \left(\L^{(m-1)}q^m\right)^{r_m} \Sym^{r_m}\left(\nicefrac{\left( \int_{X(m)}  \omega(m) \right)}{\mu_m} \right)  \right) \right)                 \nonumber     \\
&  = \prod_{m\geq 1} \left( \left(  1-\L^{m-1}q^m\right)^{-\left( \int_{X(m)}  \omega(m) \right)/\mu_m} \right)   \punto 
\end{align}

\end{proof}

\subsubsection{}  Applying this identity to every coefficient of the zeta function $Z_{X,\omega}(T)$, we obtain a formula for the motivic zeta function for its Hilbert schemes of points:

\begin{theorem}
Assume that either $\chara k=0$ or $\chara k> n$, then the following equation holds: 
\begin{equation}
\label{eqn:explZ}
Z_{\Hilb^n(X), \omega^{[n]}} = \sum_{\alpha \dashv n} \prod_{j=1}^\infty  \left(  \L^{(j-1)\alpha_j} \Sym^{\alpha_j}\left( Z_{X(j), \omega(j)} \right) \right)  \punto
\end{equation}
\end{theorem}

\begin{proof}
It follows after an application of Corollary \ref{cor:explint} to each coefficient of the zeta function.
\end{proof}


\section{Proof of the conjecture}
\label{sect:proofMC}

\subsection{Poles of the Zeta function}

\subsubsection{} Througout this section, we denote by $\overline{\mathcal{R}}$ one of the two rings $\mathcal{M}_k\left[ (\L^r-1)^{-1} \colon 0<r\in \N \right]$ or $\widehat{\mathcal{M}}_k$, while $\mathcal{R}$ will denote either $\overline{\mathcal{R}}$ or $\mathcal{M}_k$. 

The aim of this section is to study the poles of $Z_{\Hilb^n(X),\omega^{[n]}}(T)$ in terms of those of $Z_{X,\omega}$ and deduce the following:

\begin{theorem}[Monodromy conjecture for Hilbert schemes]
\label{thm:monconjHn}
Let $X$ be a surface with trivial canonical bundle satisfying the monodromy conjecture in $\overline{\mathcal{R}}$. \\
If $\chara k=0$, then the same holds for $\Hilb^{n}(X)$, $\forall n\in \N$. If $\chara k=p>0$ and $X$ admits a model as in \S \ref{subsect:constwnm}, then the monodromy conjecture in $\overline{\mathcal{R}}$ holds for $\Hilb^n(X)$, $\forall n <\chara k$.
\end{theorem} 

\subsubsection{} Despite not being the aim of our discussion, we report here the following statement, which can be obtained as a byproduct of the argments we have developed so far:

\begin{proposition}
\label{prop:monconjprod}
Let $Y, Z$ be two Calabi-Yau varieties endowed with volume forms $\omega_1, \omega_2$ satisfying the monodromy conjecture in $\overline{\mathcal{R}}$. Let $\omega$ be volume forms on $Y \times Z$ the volume form $\omega\coloneqq \pr_Y^\ast \omega_1 \wedge \pr_Z^\ast \omega_2$. Then also $Y\times Z$, endowed with the volume form $\omega$, satisfies the monodromy conjecture in $\overline{\mathcal{R}}$.
\end{proposition}


\subsubsection{} We start with a simple remark:
\begin{remark}
For an arbitrary positive integer $l>0$ we have that
$$ l\cdot Z_{X(l),\omega(l)}(T^l)= \sum_{i=0}^{l-1}Z_{X,\omega}(\zeta_l^i T) \virgola $$
where we consider the functions as power series with coefficient in an algebraic extensions of $\mathcal{R}$ containing the $l-$th roots of unity (though after the due cancellations, the above equation involves only elements of $\mathcal{R}$).
Thus, by writing $Z_{X,\omega}(T)$ in the form Equation \eqref{eqn:zetaK3form}, with $N$ divisible by $l$, we see that the set of poles of $Z_{X(l),\omega(l)}(T)$ is contained in $l\cdot \pazocal{P}$.
\end{remark}

\subsubsection{} Using this remark and the results from \S \ref{sect:studysym} we get an upper bound on the set of poles of $Z_{\Hilb^n(X),\omega^{[n]}}(T)$:

\begin{corollary}
\label{cor:polHilb}
Let $X$ be a surface with trivial canonical bundle and $\omega$ a volume form on it. Assume that $Z_{X,\omega}(T)\in \mathcal{R}[[T]]$ can be written as a sum of functions with only one pole. Let $\pazocal{P}$ be the set of poles of $Z_{X,\omega}$. Then all the poles of $Z_{\Hilb^n(X),\omega^{[n]}}(T)$ are contained in $\Sigma^n \pazocal{P}$.
\end{corollary}

\begin{proof}
Let us write $Z_{X(j),\omega(j)}/\mu_j(T)=\sum_{i\geq 0} A_i^{(j)}T^i$. For each $\alpha\dashv n$, let 
$$F_\alpha(T)\coloneqq  \sum_{i>0}\left(\Sym^{\alpha_1}A_i^{(1)}\right)\cdots\left(\Sym^{\alpha_n}A_i^{(n)}\right)T^i \punto $$
According to Equation \eqref{eqn:formulan}, we have that $Z_{\Hilb^n(X),\omega^{[n]}}(T)=\sum_{\alpha \dashv n}\L^{n-|\alpha|} F_\alpha(T)$, thus we only need to prove that $F_\alpha$ has only poles inside $\Sigma^n \pazocal{P}$. Lemma \ref{lem:polestot} implies that $(\Sym^{\alpha_j} Z_{X(j),\omega(j)}/\mu_j)(T)$ has poles in $\Sigma^{\alpha_j}(j\pazocal{P})\subseteq \Sigma^{j\alpha_j}\pazocal{P}$; our statement follows from Proposition \ref{prop:prodfun} and from the identity
$$ \Sigma^{\alpha_1}\pazocal{P} + \Sigma^{2\alpha_2}\pazocal{P}+\cdots+\Sigma^{n\alpha_n}\pazocal{P}=\Sigma^{n}\pazocal{P} \punto $$
\end{proof}

\subsubsection{} We are now ready to prove Theorem \ref{thm:monconjHn}:
\begin{proof}[Proof of Theorem \ref{thm:monconjHn}]
Let $q$ be a pole of $Z_{\Hilb^n(X),\omega^{[n]}}(T)$ and let $\sigma\in \Gal(\overline{K}|K)$ be a topological generator of the tame Galois subgroup. Consider poles  $q_1,\dots,q_n$ of $Z_{X,\omega}(T)$ such that $q=q_1+\cdots+q_n$. Since the monodromy conjecture holds for $X$, there are elements $v_1,v_2,\dots,v_n\in H^\ast(X_{\overline{K}},\Q_l)$ such that $\sigma(v_j)=e^{2\pi i q_j}v_j$.
Let us consider the Galois-equivariant isomorphism from \cite[Theorem 2]{GS}:
$$ H^\ast(\Hilb^n(X),\Q_l)\cong \bigoplus_{\alpha \dashv n} H^\ast(\Sym^{|\alpha|}(X),\Q_l)(n- |\alpha|) \puntovirgola $$
focusing on the summand $H^\ast(\Sym^n(X),\Q_l)\cong H^\ast(X^n,\Q_l)^{\Sigma_n}$, where the action of $\Sigma_n$ on \\ $H^\ast(X^n,\Q_l)\cong H^\ast(X,\Q_l)^{\otimes n}$ is induced by the usual action $\Sigma_n\curvearrowright X^n$ given by permutation of the factors. Thus the element
$$ \sum_{\rho\in \Sigma_n}v_{\rho(1)}\otimes \cdots \otimes v_{\rho(n)} $$
is a non-zero eigenvector of $H^\ast(\Hilb^n(X),\Q_l)$ for the eigenvalue $\prod_{j=1}^n e^{2\pi i q_j}$.
\end{proof}

\subsubsection{} And similarly:
\begin{proof}[Proof of Proposition \ref{prop:monconjprod}.]
We have that $\displaystyle \int_{Y(n) \times Z(n)}\omega(n)= \left(\int_{Y(n)}\omega_1(n) \right) \left( \int_{Z(n)} \omega_2(n) \right)$. 

Hence Proposition \ref{prop:prodfun} implies that $\forall q$ pole of $Z_{Y\times Z, \omega}(T)$ there are a pole $q_1$ of $Z_{Y, \omega_1}(T)$ and a pole $q_2$ of $Z_{Z, \omega_2}(T)$ such that $q=q_1+q_2$. 

Since $Y$ and $Z$ satisfy the monodromy conjecture, there are nonzero eigenvectors $v\in H^\ast(Y, \Q_l)$ with eigenvalue $\exp(2\pi i q_1)$ and $w\in H^\ast(Z, \Q_l)$ with eigenvalue $\exp(2\pi i q_2)$, so that the element $v\otimes w\in H^\ast(Y, \Q_l)\otimes H^\ast(Z, \Q_l) \cong H^\ast(Y\times Z, \Q_l)$ is an eigenvector with eigenvalue $\exp(2 \pi i q)$.
\end{proof}

\subsubsection{}  We are not able to say much about the monodromy conjecture in $\mathcal{M}_k$ for all the Hilbert schemes of points on a surface, since it is not always possible to write $Z_{X,\omega}(T)\in \mathcal{M}_k[[T]]$ as a sum of functions with a single pole. However, there are a few remarkable classes of surfaces whose zeta function has a unique pole. In these cases, such a condition
is automatically satisfied, so also $Z_{\Hilb^n(X),\omega^{[n]}}$ has a unique pole which is $n$ times the pole of $Z_{X,\omega}$ and $\Hilb^n(X)$ will then satisfy the monodromy property.

\begin{example}
We list a few classes of surfaces satisfying the property above:
\begin{itemize}
\item Assume that $X$ is an abelian surface; according to \cite{HN_Ab}, $Z_{X,\omega}$ has a unique pole which coincides with Chai's basechange conductor of $X$;
\item If $X\rightarrow \Spec K$ is a $K3$ surface admitting an equivariant Kulikov model after a finite base change with respect to afinite extension $F/K$, then Halle and Nicaise proved in \cite{HN} that $Z_{X,\omega}$ has a unique pole;
\item Assume that $X$ is a Kummer surface constructed from an abelian surface $A$; then Overkamp proved in \cite{Ove} that $Z_{X,\omega}$ has a unique pole.
\end{itemize}
Moreover all the surfaces in this list satisfy the monodromy property.
\end{example}

\begin{corollary}
Let $X$ be a surface in the list above, then the monodromy conjecture holds for $\Hilb^n(X)$, provided that either $\chara k=0$ or $\chara k >n$, with the usual assumptions on the models of $X$. 

Similarly, the monodromy conjecture holds for a product $X_1\times \dots \times X_n$, where all the $X_i$ are surfaces in the list above.
\end{corollary}

\begin{proof}
The first statement follows from \ref{thm:monconjHn}, while the latter follows from \ref{prop:monconjprod}.
\end{proof}

\subsection{Further remarks}

\subsubsection{} It is still an open question whether all the sums of $n$ poles of $Z_{X,\omega}(T)$ are actually poles of $Z_{\Hilb^n(X),\omega^{[n]}}(T)$, or if cancellation might occur. It is reasonable to expect that given a very general $K3$ surface $X$ and a positive integer $n$, if $\pazocal{P}$ is the set of poles of $Z_{X,\omega}$, then the set of poles of $Z_{Hilb^n(X),\omega^{[n]}}$ coincides with $\Sigma^n\pazocal{P}$. We expect this in virtue of the fact that the expression for $Z_{\Hilb^{n}(X),\omega^{[n]}}$, obtained by following the algorithm of Corollary \ref{cor:polHilb} and \S \ref{sect:studysym}, contains terms having a pole in each element of $\Sigma^n\pazocal{P}$ and the cancellation among them "should happen only exceptionally".

\begin{example}
Let $K\coloneqq k((t))$, $R\coloneqq k[[t]]$, where $\chara k=0$. Let $X\subseteq \P^3_K$ (with homogeneous coordinates $[w:x:y:z]$) be the surface defined by the quartic polynomial:
$$ w^2x^2+w^2y^2+w^2z^2+x^4+y^4+z^4+tw^4 \virgola $$
and let, finally, $\omega$ be an arbitrary volume form over it; this example was already studied in \cite{HN}. It is possible to prove that $Z_{X,\omega}$ has two poles and satisfies the monodromy conjecture in $\mathcal{M}_k$. A direct computation (relying on a construction we will sketch later) shows that the poles of $Z_{\Hilb^2(X),\omega^{[2]}}$ are actually the three expected poles. 
Let $\mathfrak{Y}\subseteq \P^3_R$ the model of $X$ obtained by the above equation considered as a polynomial with coefficients in $R$. The model constructed in this way is a regular model whose central fibre is an irreducible surface with only a singular point $O$ of type $A_1$. After blowing up $O\in \mathfrak{Y}$ one obtains a regular model with strict normal crossing divisor, whose central fibre consists of two components: a regular $K3$ surface (the strict transform of $\mathfrak{Y}_k$ and a copy of $\P^2_k$ with multiplicity 2, their intersection is a rational curve which sits in $\P^2$ as a conic. After semistable reduction, one gets a model $\mathfrak{X}(2)\rightarrow R(2)$ of $X(2)$ whose central fibre consists of5 a smooth $K3$ surface intersecting $\P^1\times \P^1$ along its diagonal. Using the construction of Nagai in \cite{Nag} one gets a semistable model for $\Hilb^2(X(2))$ over $R(2)$ and after basechanging such model and Weil-restricting it is possible to compute the motivic integral of $\Hilb^2(X(m))$ for all $m\in \N$ and 
thus $Z_{\Hilb^2(X),\omega^{[2]}}\in \mathcal{M}_k[[T]]$. After specializing the zeta function using the Poincaré polynomial one sees that all the three possible poles are indeed poles for $Z_{\Hilb^2(X),\omega^{[2]}}$.
\end{example}


\vfill

{\textsc{University of Copenhagen, department of Mathematical Sciences, \\
Universitetparken 5, 2100 Copenhagen, Denmark}} 

\emph{E-mail address}: luigi.pagano@math.ku.dk

\end{document}